\documentclass[12pt,reqno]{amsart}
\usepackage{color}
\usepackage[english]{babel}
\usepackage{amsfonts}
\usepackage{amsmath}
\usepackage{amsthm}
\usepackage{amssymb}
\usepackage[misc]{ifsym}
\usepackage{bbm}
\usepackage{mathrsfs}
\usepackage{esint}
\usepackage{faktor}
\usepackage[utf8]{inputenc}
\usepackage{afterpage}
\usepackage[left=2.9cm,right=2.9cm,top=2.8cm,bottom=2.8cm]{geometry}
\usepackage{graphicx}
\usepackage[pagewise,displaymath,mathlines]{lineno}
\newcommand{\N}{\mathbb{N}}

\newcommand{\R}{\mathbb{R}}

\newcommand{\Div}{\mathrm{div} \, }
\renewcommand{\epsilon}{\varepsilon}
\renewcommand{\phi}{\varphi}

\newtheorem{lemma}{Lemma}[section]
\newtheorem{thm}[lemma]{Theorem}
\newtheorem{prop}[lemma]{Proposition}

\theoremstyle{definition}
\newtheorem{defi}[lemma]{Definition}
\newtheorem{rmk}[lemma]{Remark}

\numberwithin{equation}{section}

\DeclareMathOperator*{\essinf}{ess \, inf}
\DeclareMathOperator*{\supp}{supp}

\begin{document}

\title[Strongly singular convective equations in $ \R^N $]{Strongly singular convective \\ elliptic equations in $\R^N$ driven \\ by a non-homogeneous operator}
\author[L. Gambera]{Laura Gambera}
\address[L. Gambera]{Dipartimento di Matematica e Informatica, Universit\`a degli Studi di Catania, Viale A. Doria 6, 95125 Catania, Italy}
\email{laura.gambera@studium.unict.it}
\author[U. Guarnotta]{Umberto Guarnotta}
\address[U. Guarnotta]{Dipartimento di Matematica e Informatica, Universit\`a degli Studi di Palermo, Via Archirafi 34, 90123 Palermo, Italy}
\email{umberto.guarnotta@unipa.it}

\maketitle

\begin{abstract}
Existence of a generalized solution to a strongly singular convective elliptic equation in the whole space is established. The differential operator, patterned after the $(p,q)$-Laplacian, can be non-homogeneous. The result is obtained by solving some regularized problems through fixed point theory, variational methods and compactness results, besides exploiting nonlinear regularity theory and comparison principles.
\end{abstract}

\let\thefootnote\relax
\footnote{\Letter \quad Umberto Guarnotta (umberto.guarnotta@unipa.it).}

%\begin{center}
%\begin{minipage}{9.5cm}
%\tableofcontents
%\end{minipage}
%\end{center}

\section{Introduction and main result}

In this paper we deal with the problem
\begin{equation}
\label{prob}
\tag{P}
\left\{
\begin{alignedat}{2}
-\Div a(\nabla u) &= f(x,u) + g(x,\nabla u) \quad &&\mbox{in} \;\; \R^N, \\
u &> 0 \quad &&\mbox{in} \;\; \R^N,
\end{alignedat}
\right.
\end{equation}
where $ N \geq 2 $. The differential operator $ u \mapsto \Div a(\nabla u) $, usually called $ a $-Laplacian, is patterned after the $ (p,q) $-Laplacian $ \Delta_p + \Delta_q $, $ 1<q<p<+\infty $, where $ \Delta_p u := \Div(|\nabla u|^{p-2} \nabla u) $, as usual; see $ {\rm (H_a)} $ in Section 2 for details. Hereafter we assume $ 1<p<N $. We suppose that $ f: \R^N \times (0,+\infty) \to [0,+\infty) $ and $ g: \R^N \times \R^N \to [0,+\infty) $ are Carathéodory functions satisfying the following conditions:
\begin{equation}
\label{hypf}
\tag{$ {\rm H_f} $}
\begin{split}
\liminf_{s \to 0^+} f(x,s) > 0 \quad \mbox{uniformly w.r.t.} \;\; x \in B_\sigma(x_0), \\
f(x,s) \leq h(x)s^{-\gamma}, \quad h \in L^1(\R^N) \cap L^\eta(\R^N), \quad h \geq 0,
\end{split}
\end{equation}
and
\begin{equation}
\label{hypg}
\tag{$ {\rm H_g} $}
g(x,\xi) \leq k(x)|\xi|^r, \quad k \in L^1(\R^N) \cap L^\theta(\R^N), \quad k \geq 0,
\end{equation}
for some $ x_0 \in \R^N $, $ \sigma \in (0,1) $, $ \gamma \geq 1 $, $ r \in [0,p-1) $, and
\begin{equation}
\label{hyphkr}
\eta > (p^*)', \quad \theta > \left(\frac{1}{(p^*)'}-\frac{r}{p}\right)^{-1},
\end{equation}
being $ p^* := \frac{Np}{N-p} $ the Sobolev critical exponent. In the sequel we will suppose, without loss of generality, that $ x_0 = 0 $.

Due to the strongly singular nature of the reaction term, the word `solution' has to be understood in a suitable sense. Here we adopt the definition of generalized solution used in \cite{CP}; adapted to our context, it reads as follows.
\begin{defi}
\label{gensol}
$ u \in W^{1,p}_{\rm loc}(\R^N) $ is a \textit{generalized solution} to \eqref{prob} if the following conditions hold true:
\begin{itemize}
\item[$ {\rm (i)} $] For any compact $ K \subseteq \R^N $ one has
\begin{linenomath}
\begin{equation*}
\essinf_K u > 0.
\end{equation*}
\end{linenomath}
\item[$ {\rm (ii)} $] For any $ \phi \in C^\infty_c(\R^N) $ it holds
\begin{linenomath}
\begin{equation*}
\int_{\R^N} a(\nabla u) \cdot \nabla \phi = \int_{\R^N} [f(x,u) + g(x,\nabla u)] \phi.
\end{equation*}
\end{linenomath}
\end{itemize}
\end{defi} 
Our goal is to prove the following.
\begin{thm}
\label{mainthm}
Under $ {\rm (H_a)} $, \eqref{hypf}--\eqref{hypg}, and \eqref{hyphkr}, there exists $ u \in W^{1,p}_{\rm loc}(\R^N) $ generalized solution to \eqref{prob}.
\end{thm}
Now we briefly describe the technique adopted to prove Theorem \ref{mainthm}. We consider a sequence of regularized problems \eqref{ballprob}, obtained by shifting the singular terms through a parameter $ \epsilon_n $  and working in a ball $ B_n $, with $ \epsilon_n \searrow 0 $ and $ B_n \nearrow \R^N $ as $ n \to \infty $. We want to find, for all $ n \in \N $, a solution $ u_n $ to \eqref{ballprob}; this can be done via a sub-super-solution theorem (see Theorem \ref{subsuperthm}). \\
We first construct a function $ \overline{u} \in W^{1,p}_{\rm loc}(\R^N) $ (vide \eqref{supersoldef}), which is a super-solution of each \eqref{ballprob}, independently of $ n $. The construction of $ \overline{u} $ is done via fixed point theory (Lemma \ref{schauder}) and approximation (Lemma \ref{supersollemma}), besides exploiting the compactness results presented in Lemmas \ref{gradconvlocal} and \ref{gradconvglobal}. The fact that $ \overline{u} $ is independent of $ n $ will be crucial to get the energy estimate \eqref{mainenergyest} for the sequence $ \{u_n\} $. \\
Then, for all $ n \in \N $, a regular sub-solution $ \underline{u}_n $ to \eqref{ballprob} is provided (see Lemma \ref{subsollemma}). Regularity is essential in order to control $ \underline{u}_n $ from below: this fact will allow to `avoid' the singularity of $ f $, at least on compact sets. The construction of $ \underline{u}_n $ is performed via variational methods, truncation techniques, and standard compactness results. A careful choice of the sub-solution $ \underline{u}_n $, based on Lemma \ref{subflemma} and \eqref{truncf}, allows to get \eqref{comparison}, which is a local uniform estimate from below for $ \{u_n\} $. \\
After constructing $ u_n $ with the aim of Theorem \ref{subsuperthm}, it remains to let $ n \to \infty $ and get a solution $ u $ to \eqref{prob}. This passage to the limit strongly relies on estimate \eqref{mainenergyest}, which can be obtained by a fine analysis of the energy of $ u_n $ on certain sets, which are the intersection of balls with super-level sets of $ u_n $: indeed, we can control the convection terms in the balls, while singular terms can be estimated in the super-level sets. The estimate is obtained by a localization procedure, taking into account the properties of $ \underline{u}_n $ and $ \overline{u} $. A compactness argument allows to prove that $ u $ is a generalized solution of \eqref{prob}.

Problem \eqref{prob} possesses at least four interesting peculiarities:
\begin{itemize}
\item the operator $ u \mapsto a(\nabla u) $, patterned after the $ (p,q) $-Laplacian, can be non-homogeneous;
\item the reaction term $ f $ is strongly singular (that is, $ \gamma \geq 1 $), and $ f(x,\cdot) $ can be non-monotone;
\item the reaction term $ g $ is convective (i.e., it depends on $ \nabla u $);
\item the problem is set in the whole space $ \R^N $.
\end{itemize}
Problems exhibiting some of these features arise from applications, in particular in chemistry and biology, and have been recently investigated by many authors from different points of view, as existence, regularity, and qualitative properties of solutions. Due to the large amount of publications in these areas, here we limit ourselves to synthetically recall some recent results concerning each topic.

Existence results for singular equations driven by a non-homogeneous operator have been obtained, e.g., in \cite{PW}, for Dirichlet $ (p,q) $-Laplacian problems, and \cite{GMM}, for convective $ a $-Laplacian Robin problems. Regularity of non-standard operators traces back to \cite{M,L}. Regarding existence of solutions, the lack of homogeneity of the operator prevents to use specific methods of constructing sub-solutions, as the ones used for instance in \cite{CP,CGP}, in favor of an approach similar to \cite{GMM}. On the other hand, regularity issues do not allow to work directly in the whole space, compelling to use sequence of balls tending to $ \R^N $.

The most popular singular problems are the so-called `weakly singular problems', characterized by exhibiting $ \gamma \in (0,1) $ instead of $ \gamma \geq 1 $; cf.\ \eqref{hypf}. Weakly singular problems are studied often by means of variational techniques, but this method is not directly applicable for strongly singular problems, since the solutions usually possess infinite energy, as pointed out, e.g., in \cite[p.2]{CST}. Hence, gaining compactness from energy estimates requires some efforts, as localization procedures. For a good introduction about singular problems, we refer to the monograph \cite{GR}. Restricting to strongly singular problems in the whole space, for $ p $-Laplace equations existence of radial solutions with a prescribed decay at infinity has been proved in \cite{C}. The article \cite{CP}, which represents the main motivation of our work, furnishes a solution to a $ p $-Laplace equation without exploiting any radiality or decay conditions. It is worth noticing that, in contrast to \cite{C,CP}, the function $ u \mapsto f(x,u)/s^{p-1} $ can be non-monotone and the function $ u \mapsto f(x,u) $ is not a perturbation of $ u \mapsto u^{-\gamma} $; moreover, existence of a super-solution is not postulated. The variational case, i.e. $ \gamma \in (0,1) $, is briefly sketched in Remark \ref{variational}.

The presence of convection terms destroys the variational structure of problems even simpler than \eqref{prob}. For this reason, topological methods, as the sub-super-solution technique, are widely employed in these situations; examples of application of the sub-super-solution method within suitable trapping regions are provided by \cite{LMZ,CaLiMo}, regarding singular convective equations and systems, respectively. Several papers by Ghergu and Radulescu are devoted to analyze the effects of the presence of convection terms; here we mention \cite{GR2}, which is one of the first articles treating convection equations in $ \R^N $. To the best of our knowledge, the present paper represents the first contribution about quasi-linear convective problems in the whole space.

Finally, we recall that working in $ \R^N $ causes lack of compactness of Sobolev embeddings. Nevertheless, a sub-super-solution approach is available in several contexts; see, e.g., \cite{RS}. Existence of solutions in $ \R^N $ is treated in \cite{FMT} for non-singular equations driven by the $ (p,q) $-Laplacian, while in \cite{MMM} an existence result for singular $ p $-Laplacian systems in the whole space is obtained. The very recent work \cite{GMMou} treats singular convective $ p $-Laplacian systems in the whole $ \R^N $.

An appendix concerning the $ a $-Laplacian operator concludes the article. Although this result is probably folklore, it is proved that the set of hypotheses usually employed to ensure both Lieberman's regularity theorems and Pucci-Serrin's strong maximum principle is actually equivalent to two assumptions: ellipticity and growth condition. The former hypotheses are employed in existence theory (here we mention \cite{GMP} just to give an example), while the latter conditions are widely used in regularity theory (see, for instance, \cite[Part II]{CM}).

\section{Preliminaries}

Let $ \R^N $, $ N \geq 2 $, be the $ N $-th dimensional Euclidean space endowed with the standard scalar product $ v \cdot w $ and norm $ |v| $, for any $ v,w \in \R^N $. With $ B_r(x) $ we indicate the ball of radius $ r $ centered at $ x $. Given two sets $ E,F \subseteq \R^N $, the symbol $ E \Subset F $ means $ \overline{E} \subseteq F $, where $ \overline{E} $ is the closure of $ E $; moreover, $ \chi_E $ denotes the characteristic function of $ E $, while $ |E| $ stands for its Lebesgue measure. If $ E \subseteq F $ and $ u:F \to \R $, then $ u_{\mid_E} $ represents the restriction of $ u $ to $ E $; anyway, for the sake of readability, the restriction symbol will be omitted when no confusion can arise. We set $ u_+ := \max\{u,0\} $ and $ \supp u := \overline{\{x \in \R^N: \, u(x) \neq 0 \}} $.

Suppose that $ (X,\|\cdot\|_X) $ is a Banach space and $ X^* $ is the topological dual of $ X $. We indicate with $ \langle \cdot,\cdot \rangle $ the duality brackets. If $ \{u_n\} \subseteq X $ and $ u \in X $, by $ u_n \to u $ we mean that $ \{u_n\} $ strongly converges to $ u $; weak convergence is indicated as $ u_n \rightharpoonup u $. Given another Banach space $ (Y,\|\cdot\|_Y) $, the symbol $ X \hookrightarrow Y $ denotes the continuous embedding of $ X $ into $ Y $.

We say that an operator $ T:X \to Y $ between two Banach spaces $ (X,\|\cdot\|_X) $ and $ (Y,\|\cdot\|_Y) $ is
\begin{itemize}
\item \textit{compact} if, for any bounded set $ B \subseteq X $, the set $ \overline{T(B)} $ is compact in $ Y $;
\item \textit{continuous} if $ u_n \to u $ implies $ T(u_n) \to T(u) $;
\item \textit{weakly continuous} if $ u_n \rightharpoonup u $ implies $ T(u_n) \rightharpoonup T(u) $;
\item \textit{strongly continuous} if $ u_n \rightharpoonup u $ implies $ T(u_n) \to T(u) $;
\item \textit{demi-continuous} if $ u_n \to u $ implies $ T(u_n) \rightharpoonup T(u) $;
\item \textit{completely continuous} if it is continuous and compact.
\end{itemize} 

Let $ \Omega \subseteq \R^N $ be a domain and $ X(\Omega) $ be a function space on $ \Omega $. The symbol $ X_{\rm loc}(\Omega) $ denotes the set of functions $ u: \Omega \to \R $ such that, for any compact $ K \Subset \Omega $, one has $ u_{\mid_K} \in X(K) $. A sequence $ \{u_n\} \subseteq X_{\rm loc}(\Omega) $ is bounded in $ X_{\rm loc}(\Omega) $ if $ \{u_{n \mid_K}\} $ is bounded in $ X(K) $ for any compact $ K \Subset \Omega $.

Given any domain $ \Omega \subseteq \R^N $, the set $ C^\infty_c(\Omega) $ consists of the test functions which are compactly supported in $ \Omega $. The spaces $ L^p(\Omega) $ and $ W^{1,p}(\Omega) $ are, respectively, the Lebesgue and Sobolev spaces, equipped with their standard norms. The space $ W^{1,p}_0(\Omega) $ contains exactly the functions of $ u \in W^{1,p}(\Omega) $ such that $ u = 0 $ on $ \partial \Omega $ in the sense of traces; if $ \Omega $ is bounded, we endow $ W^{1,p}_0(\Omega) $ with the standard equivalent norm given by the Poincaré inequality, that is,
\begin{equation}
\label{sobolevnorm}
\|u\|_{W^{1,p}_0(\Omega)} := \|\nabla u\|_{L^p(\Omega)}.
\end{equation}
For a generic domain $ \Omega $, the Beppo Levi spaces (or homogeneous Sobolev spaces) are defined as
\begin{linenomath}
\begin{equation*}
\mathcal{D}^{1,p}_0(\Omega) := \{u \in L^{p^*}(\Omega): \, |\nabla u| \in L^p(\Omega)\},
\end{equation*}
\end{linenomath}
and are equipped with the same norm of \eqref{sobolevnorm} (even if $ \Omega $ is unbounded). The space of distributions in $ \Omega $ will be indicated with $ \mathscr{D}'(\Omega) $.

The letter $ c $ denotes a generic positive constant depending on the data of the problem. If $ c $ depends only on a parameter, say $ p $, we write $ c_p $ instead of $ c $.

Now we list the assumptions on the differential operator $ u \mapsto a(\nabla u) $ and discuss its main properties. \\
We suppose that $ a $ has Uhlenbeck structure, i.e., $ a(\xi) = a_0(|\xi|)\xi $ for a suitable $ C^1 $ function $ a_0: (0,+\infty) \to (0,+\infty) $. Moreover, we assume the following structure hypotheses on $ a_0 $:
\begin{itemize}
\item[$ {\rm (H_a)}_1 $] \makebox[\textwidth][c]{$ mt^{p-1} \leq ta_0(t) \leq M(t^{p-1} + 1) \quad \forall t \in (0,+\infty), $}
\item[$ {\rm (H_a)}_2 $] \makebox[\textwidth][c]{$ \displaystyle{-1 < i_a := \inf_{t > 0} \frac{ta_0'(t)}{a_0(t)} \leq \sup_{t > 0} \frac{ta_0'(t)}{a_0(t)} =: s_a < +\infty,} $}
\end{itemize}
being $ 1 < p < N $ and $ M,m > 0 $ suitable constants. We denote by $ {\rm (H_a)} $ the set of assumptions $ {\rm (H_a)}_1 $--$ {\rm (H_a)}_2 $. Observe that $ {\rm (H_a)} $ implies
\begin{itemize}
\item[$ {\rm (a_1)} $] \makebox[\textwidth][c]{$ |a(\xi)| \leq M(|\xi|^{p-1} + 1) \quad \forall \xi \in \R^N, $}
\item[$ {\rm (a_2)} $] \makebox[\textwidth][c]{$ a(\xi) \cdot \xi \geq m|\xi|^p \quad \forall \xi \in \R^N, $}
\item[$ {\rm (a_3)} $] \makebox[\textwidth][c]{$ (a(\xi_1)-a(\xi_2)) \cdot (\xi_1-\xi_2) > 0 \quad \forall \xi_1,\xi_2 \in \R^N, \quad \xi_1 \neq \xi_2. $}
\end{itemize}
Incidentally, notice that $ {\rm (H_a)_1} $ is a growth condition which allows to work in the functional setting of Sobolev spaces; on the other hand, $ {\rm (H_a)}_2 $ is equivalent to the ellipticity of the differential operator.
\begin{rmk}
\label{standardrmk}
We explicitly notice that, for operators with Uhlenbeck structure, hypothesis $ {\rm (H_a)} $ is equivalent to Assumption 2.5 of \cite{GMM}; see the Appendix for details. In particular, the operator $ u \mapsto a(\nabla u) $ satisfies the assuptions of both Lieberman's nonlinear regularity theory \cite[p.320]{L} and Pucci-Serrin's strong maximum principle \cite[pp.3-5]{PS}, so that any solution to
\begin{linenomath}
\begin{equation*}
\left\{
\begin{alignedat}{2}
-\Div a(\nabla u) &= f(x) \quad &&\mbox{in} \;\; \Omega, \\
u &= 0 \quad &&\mbox{on} \;\; \partial \Omega,
\end{alignedat}
\right.
\end{equation*}
\end{linenomath}
being $ f \in L^\infty(\Omega) $ with $ f \geq  0 $ a.e.\ in $ \Omega $, fulfills $ u \in C^{1,\tau}(\overline{\Omega}) $ and $ u > 0 $ in $ \Omega $. More precisely, there exist $ \tau \in (0,1] $ and $ C = C(\|f\|_\infty) > 0 $ such that
\begin{linenomath}
\begin{equation*}
\|u\|_{C^{1,\tau}(\overline{\Omega})} \leq C.
\end{equation*}
\end{linenomath}
Nevertheless, the operator $ \Phi: W^{1,p}_0(\Omega) \to W^{-1,p'}(\Omega) $ defined as
\begin{linenomath}
\begin{equation*}
\langle \Phi(u),v \rangle = \int_\Omega a(\nabla u) \cdot \nabla v
\end{equation*}
\end{linenomath}
satisfies the Leray-Lions conditions (see, e.g., $ {\rm (H_1)} $--$ {\rm (H_3)} $ of \cite[p.43]{CLM}).
\end{rmk}
%
%For the reader's convenience, we recall the following theorem of passage to the limit under the integral sign.
%%
%\begin{prop}[Pratt's lemma]
%\label{pratt}
%Let $ (X,\mu,\mathscr{F}) $ be a measure space. Suppose $ \{f_n\} $, $ \{g_n\} $, $ \{h_n\} $ to be sequences of measurable functions such that
%%
%\begin{equation*}
%\begin{alignedat}{2}
%&f_n \to f \quad \mu\mbox{-a.e.\ in} \;\; X, \quad &&\lim_{n \to \infty} \int_X f_n \, {\rm d}\mu = \int_X f \, {\rm d}\mu \in \R, \\
%&h_n \to h \quad \mu\mbox{-a.e.\ in} \;\; X, \quad \quad &&\lim_{n \to \infty} \int_X h_n \, {\rm d}\mu = \int_X h \, {\rm d}\mu \in \R, \\
%&g_n \to g \quad \mu\mbox{-a.e.\ in} \;\; X, \quad &&f_n \leq g_n \leq h_n \quad \mu\mbox{-a.e.\ in} \;\; X.
%\end{alignedat}
%\end{equation*}
%%
%Then
%%
%\begin{equation*}
%\lim_{n \to \infty} \int_X g_n \, {\rm d}\mu = \int_X g \, {\rm d}\mu \in \R.
%\end{equation*}
%\end{prop}
%%
%\begin{proof}
%It suffices to apply Fatou's lemma to both $ g_n-f_n $ and $ h_n-g_n $, and then let $ n \to \infty $.
%\end{proof}
%%
Now we present some tools that will be used in the sequel. The following result (cf.\ \cite[Exercise 4.16]{B}) furnishes a sufficient condition for weak convergence in Lebesgue spaces.
\begin{prop}
\label{Brezis}
Let $ \Omega \subseteq \R^N $ be a measurable set, $ 1<p<+\infty $, and $ \{f_n\} \subseteq L^p(\Omega) $. Suppose that $ \{f_n\} $ is bounded in $ L^p(\Omega) $ and $ f_n \to f $ a.e.\ in $ \Omega $ for some measurable function $ f: \Omega \to \R $. Then $ f \in L^p(\Omega) $ and $ f_n \rightharpoonup f $ in $ L^p(\Omega) $.
\end{prop}
Our argument relies on the sub-super-solution theorem below (Theorem \ref{subsuperthm}), whose proof is based on the $ {\rm (S_+)} $ property and on the theory of pseudo-monotone operators.
\begin{defi}
Let $ (X,\|\cdot\|) $ be a Banach space. An operator $ A: X \to X^* $ is
\begin{itemize}
\item \textit{of type $ {\rm (S_+)} $} if, for any $ \{u_n\} \subseteq X $ and $ u \in X $ such that
\begin{equation}
\label{S+pseudohyp}
u_n \rightharpoonup u \quad \mbox{in} \;\; X, \quad \limsup_{n \to \infty} \langle A(u_n),u_n-u \rangle \leq 0,
\end{equation}
one has $ u_n \to u $ in $ X $;
\item \textit{pseudo-monotone} if, for any $ \{u_n\} \subseteq X $ and $ u \in X $ satisfying \eqref{S+pseudohyp}, one has
\begin{linenomath}
\begin{equation*}
\langle A(u), u-v \rangle \leq \liminf_{n \to \infty} \langle A(u_n),u_n-v \rangle  \quad \forall v \in X.
\end{equation*}
\end{linenomath}
\end{itemize}
\end{defi}
\begin{rmk}
\label{S+pseudormk}
It is readily seen that any continuous operator of type $ {\rm (S_+)} $ is pseudo-monotone (demi-continuity of $ A $ suffices; see, e.g., \cite[Lemma 6.7]{F}).
\end{rmk}
\begin{lemma}
\label{S+}
Let $ (X,\|\cdot\|) $ be a Banach space. Suppose that $ A: X \to X^* $ is a $ {\rm (S_+)} $ operator, while $ B:X \to X^* $ is a compact operator. Then $ C := A+B $ is a $ {\rm (S_+)} $ operator.
\end{lemma}
\begin{proof}
Let $ \{u_n\} \subseteq X $ and $ u \in X $ such that $ u_n \rightharpoonup u $ in $ X $ and
\begin{equation}
\label{limsupS+}
\limsup_{n \to \infty} \langle C(u_n), u_n-u \rangle \leq 0.
\end{equation}
Up to subsequences, compactness of $ B $ produces $ B(u_n) \to \psi $ in $ X^* $ for some $ \psi \in X^* $. So \eqref{limsupS+} reads as
\begin{linenomath}
\begin{equation*}
0 \geq \limsup_{n \to \infty} \langle A(u_n), u_n-u \rangle + \lim_{n \to \infty} \langle B(u_n), u_n-u \rangle = \limsup_{n \to \infty} \langle A(u_n), u_n-u \rangle.
\end{equation*}
\end{linenomath}
Hence, the $ {\rm (S_+)} $ property of $ A $ gives $ u_n \to u $ in $ X $, as desired.
\end{proof}
\begin{defi}
Let $ \Omega \subseteq \R^N $ be a bounded domain. Consider the problem 
\begin{equation}
\label{subsuperprob}
\left\{
\begin{alignedat}{2}
-\Div a(\nabla u) &= f(x,u,\nabla u) \quad &&\mbox{in} \;\; \Omega, \\
u &= 0 \quad &&\mbox{on} \;\; \partial \Omega,
\end{alignedat}
\right.
\end{equation}
where $ a $ satisfies $ {\rm (a_1)} $--$ {\rm (a_3)} $ and $ f: \Omega \times \R \times \R^N \to \R $ is a Carathéodory function. A function $ \underline{u} \in W^{1,p}(\Omega) $ is a \textit{sub-solution} to \eqref{subsuperprob} if $ \underline{u} \leq 0 $ on $ \partial \Omega $ in the sense of traces and
\begin{linenomath}
\begin{equation*}
\int_\Omega a(\nabla \underline{u}) \cdot \nabla \phi \leq \int_\Omega f(x,\underline{u},\nabla \underline{u}) \phi
\end{equation*}
\end{linenomath}
for all $ \phi \in W^{1,p}_0(\Omega) $ such that $ \phi \geq 0 $ a.e.\ in $ \Omega $. Analogously, $ \overline{u} \in W^{1,p}(\Omega) $ is a \textit{super-solution} to \eqref{subsuperprob} if $ \overline{u} \geq 0 $ on $ \partial \Omega $ in the sense of traces and
\begin{linenomath}
\begin{equation*}
\int_\Omega a(\nabla \overline{u}) \cdot \nabla \phi \geq \int_\Omega f(x,\overline{u},\nabla \overline{u}) \phi
\end{equation*}
\end{linenomath}
for all $ \phi \in W^{1,p}_0(\Omega) $ such that $ \phi \geq 0 $ a.e.\ in $ \Omega $.
\end{defi}
\begin{thm}[Sub-super-solution theorem]
\label{subsuperthm}
Let $ \Omega \subseteq \R^N $ be a bounded domain, and let $ \underline{u},\overline{u} \in W^{1,p}(\Omega) $ be a sub- and a super-solution of \eqref{subsuperprob}, respectively. Suppose $ \underline{u} \leq \overline{u} $ a.e.\ in $ \Omega $ and the following local growth condition for $ f $:
\begin{equation}
\label{localgrowth}
\begin{split}
|f(x,s,\xi)| \leq h(x) + k(x)|\xi|^r \quad \quad \quad \\
\mbox{for a.a.} \;\; x \in \Omega, \;\; \forall s \in [\underline{u}(x),\overline{u}(x)], \;\; \forall \xi \in \R^N,
\end{split}
\end{equation}
where $ h \in L^\eta(\Omega) $, $ k \in L^\theta(\Omega) $, and \eqref{hyphkr} holds true. Then there exists $ u \in W^{1,p}_0(\Omega) $ weak solution to \eqref{subsuperprob}.
\end{thm}
\begin{proof}
According to \eqref{hyphkr}, there exists $ s \in (1,p^*) $ such that
\begin{equation}
\label{s}
\frac{1}{\eta} < \frac{1}{s'} \quad \mbox{and} \quad \frac{1}{\theta} + \frac{r}{p} < \frac{1}{s'}.
\end{equation}
Consider the truncation operator $ T: W^{1,p}_0(\Omega) \to W^{1,p}_0(\Omega) $,
\begin{linenomath}
\begin{equation*}
T(u)(x) := \left\{
\begin{array}{ll}
\underline{u}(x) \quad &\mbox{if} \;\; u(x) < \underline{u}(x), \\
u(x) \quad &\mbox{if} \;\; \underline{u}(x) \leq u(x) \leq \overline{u}(x), \\
\overline{u}(x) \quad &\mbox{if} \;\; u(x) > \overline{u}(x),
\end{array}
\right.
\end{equation*}
\end{linenomath}
which is well defined since $ \underline{u} \leq 0 \leq \overline{u} $ on $ \partial \Omega $. Moreover, recalling \eqref{localgrowth}--\eqref{s}, define the operator $ N_T: W^{1,p}_0(\Omega) \to L^{s'}(\Omega) $ as
\begin{linenomath}
\begin{equation*}
N_T(u) := f(x,T(u),\nabla T(u)).
\end{equation*}
\end{linenomath}
Let $ i: W^{1,p}_0(\Omega) \to L^s(\Omega) $ be the canonical embedding and let $ i^*: L^{s'}(\Omega) \to W^{-1,p'}(\Omega) $ be its adjoint. Then consider $ F: W^{1,p}_0(\Omega) \to W^{-1,p'}(\Omega) $ defined as
\begin{linenomath}
\begin{equation*}
F := i^* \circ N_T.
\end{equation*}
\end{linenomath}
By \eqref{localgrowth}--\eqref{s} and H\"older's inequality, for any $ u \in W^{1,p}_0(\Omega) $ we have
\begin{equation}
\label{nemytskii}
\begin{split}
\|N_T(u)\|_{L^{s'}(\Omega)}^{s'} &\leq \int_\Omega (h+k|\nabla T(u)|^r)^{s'} \leq c_s \int_\Omega (h^{s'} + k^{s'}|\nabla u|^{rs'}) \\
&\leq c(\|h\|_{L^\eta(\Omega)}^{s'} + \|k\|_{L^\theta(\Omega)}^{s'} \|\nabla u\|_{L^p(\Omega)}^{rs'}),
\end{split}
\end{equation}
being $ c = c(|\Omega|,p,r,s,\eta,\theta) > 0 $ a suitable constant; thus $ N_T $ is bounded. To show the continuity of $ N_T $, let $ \{u_n\} \subseteq W^{1,p}_0(\Omega) $, $ u \in W^{1,p}_0(\Omega) $ such that $ u_n \to u $ in $ W^{1,p}_0(\Omega) $, and reason up to subsequences. According to \cite[Theorem 4.9]{B}, there exists $ U \in L^p(\Omega) $ such that
\begin{equation}
\label{reverselebesgue}
\begin{alignedat}{2}
&\nabla u_n(x) \to \nabla u(x) \quad &&\mbox{for a.a.} \;\; x \in \Omega, \\
&|\nabla u_n(x)| \leq U(x) \quad &&\mbox{for a.a.} \;\; x \in \Omega, \quad \forall n \in \N.
\end{alignedat}
\end{equation}
Using \eqref{localgrowth}, \eqref{reverselebesgue}, \eqref{s}, and Young's inequality we get
\begin{linenomath}
\begin{equation*}
\begin{split}
|N_T(u_n)-N_T(u)|^{s'} &\leq (|N_T(u_n)|+|N_T(u)|)^{s'} \\
&\leq c_s(h^{s'} + k^{s'} U^{rs'}) \\
&\leq c(h^\eta + k^\theta + U^p + 1) \in L^1(\Omega),
\end{split}
\end{equation*}
\end{linenomath}
being $ c = c(|\Omega|,p,r,s,\eta,\theta) > 0 $. Then Lebesgue's dominated convergence theorem and \eqref{reverselebesgue} yield
\begin{linenomath}
\begin{equation*}
\int_\Omega |N_T(u_n)-N_T(u)|^{s'} \to 0,
\end{equation*}
\end{linenomath}
proving that $ N_T $ is continuous. Summarizing, $ N_T $ is bounded and continuous. Rellich-Kondrachov's theorem ensures that $ i $ is completely continuous, so Schauder's theorem (vide \cite{K}) guarantees that $ i^* $ enjoys the same property. Hence $ F $ is a completely continuous operator; in particular, it is bounded, continuous, and compact.

Now consider the $ \Phi $ defined in Remark \ref{standardrmk}. According to \cite[Theorem 2.109]{CLM}, we have that $ \Phi $ is a bounded, continuous, $ {\rm (S_+)} $ operator. Lemma \ref{S+} then ensures that $ \Phi - F $ is a bounded, continuous, $ {\rm (S_+)} $ operator, which in turn implies that it is pseudo-monotone, by virtue of Remark \ref{S+pseudormk}. Moreover, we observe that $ \Phi - F $ is coercive: indeed, from $ {\rm (a_2)} $, \eqref{nemytskii}, and \eqref{hyphkr}, besides exploiting H\"older's, Sobolev's, and Young's inequalities, we infer
\begin{linenomath}
\begin{equation*}
\begin{split}
\langle \Phi(u)-F(u), u \rangle &\geq m \|\nabla u\|_{L^p(\Omega)}^p - \|N_T(u)\|_{L^{(p^*)'}(\Omega)} \|u\|_{L^{p^*}(\Omega)} \\
&\geq m \|\nabla u\|_{L^p(\Omega)}^p - c \|N_T(u)\|_{L^{s'}(\Omega)} \|\nabla u\|_{L^p(\Omega)} \\
&\geq m \|\nabla u\|_{L^p(\Omega)}^p - c (\|h\|_{L^\eta(\Omega)}\|\nabla u\|_{L^p(\Omega)} + \|k\|_{L^\theta(\Omega)} \|\nabla u\|_{L^p(\Omega)}^{r+1}) \\
&\geq \frac{m}{2} \|\nabla u\|_{L^p(\Omega)}^p - c(\|h\|_{L^\eta(\Omega)}^{p'} + \|k\|_{L^\theta(\Omega)}^{\frac{p}{p-r-1}}),
\end{split}
\end{equation*}
\end{linenomath}
being $ c = c(|\Omega|,N,p,r,s,\eta,\theta) > 0 $ an opportune constant changing its value at each passage. Hence, the main theorem on pseudo-monotone operators \cite[Theorem 2.9]{CLM} furnishes $ u \in W^{1,p}_0(\Omega) $ such that
\begin{equation}
\label{comp1}
\Phi(u) - F(u) = 0 \quad \mbox{in} \;\; W^{-1,p'}(\Omega).
\end{equation}
Since $ \overline{u} $ is a super-solution to \eqref{subsuperprob}, it satisfies
\begin{equation}
\label{comp2}
\Phi(\overline{u}) - F(\overline{u}) \geq 0 \quad \mbox{in} \;\; W^{-1,p'}(\Omega).
\end{equation}
Subtracting \eqref{comp2} from \eqref{comp1} and testing with $ \phi = (u-\overline{u})_+ \in W^{1,p}_0(\Omega) $ leads to
\begin{equation}
\label{comp3}
\langle \Phi(u)-\Phi(\overline{u}), (u-\overline{u})_+ \rangle - \langle F(u)-F(\overline{u}), (u-\overline{u})_+ \rangle \leq 0.
\end{equation}
By definition of $ F $ we have
\begin{linenomath}
\begin{equation*}
\begin{split}
&\langle F(u)-F(\overline{u}), (u-\overline{u})_+ \rangle \\
&= \int_{\Omega \cap \{u > \overline{u}\}} [f(x,T(u),\nabla T(u)) - f(x,\overline{u},\nabla \overline{u})](u-\overline{u})_+ =  0.
\end{split}
\end{equation*}
\end{linenomath}
On the other hand, $ {\rm (a_3)} $ produces
\begin{linenomath}
\begin{equation*}
\langle \Phi(u)-\Phi(\overline{u}), (u-\overline{u})_+ \rangle = \int_{\Omega \cap \{u > \overline{u}\}} (a(\nabla u) - a(\nabla \overline{u})) \cdot (\nabla u - \nabla \overline{u}) \geq 0.
\end{equation*}
\end{linenomath}
Thus, through \eqref{comp3} and $ {\rm (a_3)} $, we get $ \nabla u = \nabla \overline{u} $ a.e.\ in $ \Omega \cap \{u > \overline{u}\} $, which implies $ \nabla \phi = 0 $ a.e.\ in $ \Omega $. Since $ \phi \in W^{1,p}_0(\Omega) $, we obtain $ \phi = 0 $ a.e.\ in $ \Omega $, whence $ u \leq \overline{u} $ a.e.\ in $ \Omega $. An analogous argument yields $ u \geq \underline{u} $ a.e.\ in $ \Omega $, concluding the proof.
\end{proof}
At the end of this preliminary section, we present two results about $ L^p $ convergence of gradients of solutions to elliptic equations; they rely on the corresponding pointwise counterpart, which is investigated, in a more general setting, in \cite{BM}. The first result has local nature.
\begin{lemma}
\label{gradconvlocal}
Let $ \Omega \subseteq \R^N $ be a domain and $ 1<s<p^* $. Let $ \{f_n\} \subseteq L^{s'}_{\rm loc}(\Omega) $ and $ \{u_n\} \subseteq W^{1,p}_{\rm loc}(\Omega) $ be such that
\begin{equation}
\label{seqprob}
-\Div a(\nabla u_n) = f_n(x) \quad \mbox{in} \;\; \mathscr{D}'(\Omega),
\end{equation}
where $ a $ satisfies $ {\rm (a_1)} $--$ {\rm (a_3)} $. Suppose $ \{f_n\} $ and $ \{u_n\} $ to be bounded, respectively, in $ L^{s'}_{\rm loc}(\Omega) $ and $ W^{1,p}_{\rm loc}(\Omega) $. Then there exists $ f \in L^{s'}_{\rm loc}(\Omega) $ and $ u \in W^{1,p}_{\rm loc}(\Omega) $ such that, up to subsequences, $ f_n \rightharpoonup f $ in $ L^{s'}(E) $ and $ u_n \to u $ in $ W^{1,p}(E) $ for any bounded domain $ E \Subset \Omega $. In particular, $ u $ solves
\begin{equation}
\label{limitprob}
-\Div a(\nabla u) = f(x) \quad \mbox{in} \;\; \mathscr{D}'(\Omega).
\end{equation}
\end{lemma}
\begin{proof}
\phantom{.} \\
\textit{\underline{Step 1}: pointwise convergence of gradients.} \\
Take any bounded domain $ E \Subset \Omega $ and reason up to subsequences. By reflexivity we have
\begin{equation}
\label{boccardo1}
u_n \rightharpoonup u \quad \mbox{in} \;\; W^{1,p}(E) \quad \mbox{and} \quad f_n \rightharpoonup f \quad \mbox{in} \;\; L^{s'}(E),
\end{equation}
for suitable $ u \in W^{1,p}(E) $ and $ f \in L^{s'}(E) $. Exploiting the compactness of the embeddings $ W^{1,p}(E) \hookrightarrow L^s(E) $, guaranteed by Rellich-Kondrachov's theorem, and $ L^{s'}(E) \hookrightarrow W^{-1,p'}(E) $, achieved via Schauder's theorem (see \cite{K}), we get
\begin{equation}
\label{boccardo2}
u_n \to u \quad \mbox{in} \;\; L^p(E) \quad \mbox{and} \quad f_n \to f \quad \mbox{in} \;\; W^{-1,p'}(E).
\end{equation}
By \eqref{boccardo1}--\eqref{boccardo2}, besides recalling Remark \ref{standardrmk}, we are in the position to apply \cite[Theorem 2.1]{BM}, which ensures that
\begin{equation}
\label{boccardo3}
\nabla u_n \to \nabla u \quad \mbox{a.e.\ in} \;\; E.
\end{equation}
Since $ E $ was arbitrary, a diagonal argument allows us to suppose $ f \in L^{s'}_{\rm loc}(\Omega) $ and $ u \in W^{1,p}_{\rm loc}(\Omega) $ in \eqref{boccardo1}--\eqref{boccardo2}, as well as
\begin{equation}
\label{pointconvgrad}
\nabla u_n \to \nabla u \quad \mbox{a.e.\ in} \;\; \Omega.
\end{equation}
Indeed, let us consider a sequence $ \{E_n\} $ of bounded domains such that $ E_n \Subset \Omega $ and $ E_n \nearrow \Omega $. According to \eqref{boccardo1}--\eqref{boccardo2}, there exist $ \{u_{k^{(1)}_n}\} $, subsequence of $ \{u_n\} $, and $ u^{(1)} $ such that
\begin{equation}
\label{diagonalu1}
\begin{split}
&u_{k^{(1)}_n} \rightharpoonup u^{(1)} \quad \mbox{in} \;\; W^{1,p}(E_1), \\
&u_{k^{(1)}_n} \to u^{(1)} \quad \mbox{in} \;\; L^p(E_1), \\
&u_{k^{(1)}_n} \to u^{(1)} \quad \mbox{a.e.\ in} \;\; E_1.
\end{split}
\end{equation}
By induction, for all $ j \in \N $, $ j > 1 $, there exist $ \{u_{k^{(j)}_n}\} $, subsequence of $ \{u_{k^{(j-1)}_n}\} $, and $ u^{(j)} $ such that
\begin{equation}
\label{diagonalu2}
\begin{split}
&u_{k^{(j)}_n} \rightharpoonup u^{(j)} \quad \mbox{in} \;\; W^{1,p}(E_j), \\
&u_{k^{(j)}_n} \to u^{(j)} \quad \mbox{in} \;\; L^p(E_j), \\
&u_{k^{(j)}_n} \to u^{(j)} \quad \mbox{a.e.\ in} \;\; E_j.
\end{split}
\end{equation}
Define $ u: \Omega \to \R $ as
\begin{linenomath}
\begin{equation*}
u(x) = \left\{
\begin{array}{ll}
u^{(1)}(x) \quad &\mbox{in} \;\; E_1, \\
u^{(j)}(x) \quad &\mbox{in} \;\; E_j \setminus E_{j-1}, \;\; \forall j > 1.
\end{array}
\right.
\end{equation*}
\end{linenomath}
By uniqueness of limit and \eqref{diagonalu1}--\eqref{diagonalu2}, one has $ u = u^{(j)} $ a.e.\ in $ E_j $ for all $ j \in \N $. Recalling that $ \{u_{k^{(n)}_n}\} $ is a subsequence of $ \{u_{k^{(j)}_n}\} $ for all $ j \in \N $, by \eqref{diagonalu1}--\eqref{diagonalu2} we get
\begin{linenomath}
\begin{equation*}
u_{k^{(n)}_n} \rightharpoonup u \quad \mbox{in} \;\; W^{1,p}(E_j) \quad \mbox{and} \quad u_{k^{(n)}_n} \to u \quad \mbox{in} \;\; L^p(E_j) \quad \forall j \in \N.
\end{equation*}
\end{linenomath}
Reasoning as before, \eqref{boccardo1}--\eqref{boccardo2} entail
\begin{equation}
\label{diagonalf}
f_{k^{(j)}_n} \rightharpoonup f^{(j)} \quad \mbox{in} \;\; L^{s'}(E_j) \quad \mbox{and} \quad f_{k^{(j)}_n} \to f^{(j)} \quad \mbox{in} \;\; W^{-1,p'}(E_j) \quad \forall j \in \N,
\end{equation}
for opportune $ \{k^{(j)}_n\} $ (extract from $ \{k^{(j-1)}_n\} $) and $ f^{(j)} \in L^{s'}(E_j) $. Define $ f: \Omega \to \R $ as
\begin{linenomath}
\begin{equation*}
f(x) = \left\{
\begin{array}{ll}
f^{(1)}(x) \quad &\mbox{in} \;\; E_1, \\
f^{(j)}(x) \quad &\mbox{in} \;\; E_j \setminus E_{j-1}, \;\; \forall j > 1.
\end{array}
\right.
\end{equation*}
\end{linenomath}
Pick $ i,j \in \N $ such that $ i < j $. For any $ g \in L^s(E_i) $ consider its extension $ \hat{g} \in L^s(E_j) $ defined by setting $ g \equiv 0 $ outside $ E_i $. We obtain, due to \eqref{diagonalf},
\begin{linenomath}
\begin{equation*}
\int_{E_i} gf^{(j)} = \int_{E_j} \hat{g}f^{(j)} = \lim_{n \to \infty} \int_{E_j} \hat{g}f_{k^{(n)}_n} = \lim_{n \to \infty} \int_{E_i} gf_{k^{(n)}_n} = \int_{E_i} gf^{(i)} \quad \forall g \in L^s(E_i).
\end{equation*}
\end{linenomath}
This implies $ f^{(i)} = f^{(j)} $ a.e.\ in $ E_i $, whence $ f = f^{(j)} $ a.e.\ in $ E_j $. In particular, by \eqref{diagonalf} we infer
\begin{linenomath}
\begin{equation*}
f_{k^{(n)}_n} \rightharpoonup f \quad \mbox{in} \;\; L^{s'}(E_j) \quad \mbox{and} \quad f_{k^{(n)}_n} \to f \quad \mbox{in} \;\; W^{-1,p'}(E_j) \quad \forall j \in \N.
\end{equation*}
\end{linenomath}
Moreover \eqref{boccardo3}, written for each $ E_j $, readily implies \eqref{pointconvgrad}.

\noindent
\textit{\underline{Step 2}: alternative formulations of \eqref{seqprob}--\eqref{limitprob}.} \\
Notice that $ {\rm (a_1)} $ produces
\begin{equation}
\label{bounda}
\int_E |a(\nabla u_n)|^{p'} \leq M^{p'} \int_E (|\nabla u_n|^{p-1} + 1)^{p'} \leq c_p M^{p'} (\|\nabla u_n\|_{L^p(E)}^p + |E|).
\end{equation}
Thus Proposition \ref{Brezis}, together with \eqref{boccardo1} and \eqref{boccardo3}, implies
\begin{equation}
\label{weakconv}
a(\nabla u_n) \rightharpoonup a(\nabla u) \quad \mbox{in} \;\; L^{p'}(E).
\end{equation}
A density argument applied to \eqref{seqprob} yields
\begin{equation}
\label{test1}
\int_\Omega a(\nabla u_n) \cdot \nabla \phi = \int_\Omega f_n \phi \quad \forall \phi \in W^{1,p}_0(\Omega), \quad \supp \phi \Subset E.
\end{equation}
Letting $ n \to \infty $ in \eqref{test1}, via \eqref{weakconv} and \eqref{boccardo1}, produces
\begin{equation}
\label{test2}
\int_\Omega a(\nabla u) \cdot \nabla \phi = \int_\Omega f \phi \quad \forall \phi \in W^{1,p}_0(\Omega), \quad \supp \phi \Subset E,
\end{equation}
which readily implies \eqref{limitprob}.

\noindent
\textit{\underline{Step 3}: convergence of local energy integrals.} \\
Fix any $ \psi \in C^\infty_c(\Omega) $ and take a bounded domain $ E $ such that $ \supp \psi \Subset E \Subset \Omega $. Choosing $ \phi = u_n \psi $ in \eqref{test1} and $ \phi = u\psi $ in \eqref{test2} produces
\begin{equation}
\label{test3}
\begin{split}
&\int_\Omega \psi a(\nabla u_n) \cdot \nabla u_n + \int_\Omega u_n a(\nabla u_n) \cdot \nabla \psi = \int_\Omega f_n u_n \psi, \\
&\int_\Omega \psi a(\nabla u) \cdot \nabla u + \int_\Omega u a(\nabla u) \cdot \nabla \psi = \int_\Omega f u \psi.
\end{split}
\end{equation}
Reasoning as in \eqref{bounda}--\eqref{weakconv} we obtain
\begin{linenomath}
\begin{equation*}
\begin{split}
\|a(\nabla u_n) \cdot \nabla \psi\|_{L^{p'}(E)}^{p'} &\leq \|\nabla \psi\|_\infty^{p'} M^{p'} \int_E (|\nabla u_n|^{p-1}+1)^{p'} \\
&\leq c (\|\nabla u_n\|_{L^p(E)}^p + 1)
\end{split}
\end{equation*}
\end{linenomath}
for a suitable $ c = c(p,M,\psi,|E|) > 0 $, whence
\begin{equation}
\label{weakconv2}
a(\nabla u_n) \cdot \nabla \psi \rightharpoonup a(\nabla u) \cdot \nabla \psi \quad \mbox{in} \;\; L^{p'}(E).
\end{equation}
Using \eqref{boccardo2} and \eqref{weakconv2} we get
\begin{equation}
\label{conv1}
\lim_{n \to \infty} \int_\Omega u_n a(\nabla u_n) \cdot \nabla \psi = \int_\Omega u a(\nabla u) \cdot \nabla \psi.
\end{equation}
Observe that \eqref{boccardo1} and the Rellich-Kondrachov theorem entail
\begin{linenomath}
\begin{equation*}
\|u_n\psi - u\psi\|_{L^s(E)} \leq \|\psi\|_{L^\infty(E)} \|u_n-u\|_{L^s(E)} \to 0,
\end{equation*}
\end{linenomath}
so
\begin{equation}
\label{conv2}
\lim_{n \to \infty} \int_\Omega f_n u_n \psi = \int_\Omega f u \psi,
\end{equation}
using \eqref{boccardo1} again. Subtracting \eqref{test3} term by term and letting $ n \to \infty $, besides using \eqref{conv1}--\eqref{conv2}, yield
\begin{equation}
\label{conv3}
\lim_{n \to \infty} \int_\Omega \psi a(\nabla u_n) \cdot \nabla u_n = \int_\Omega \psi a(\nabla u) \cdot \nabla u.
\end{equation}

\noindent
\textit{\underline{Step 4}: local convergence of $ \{u_n \} $ to $ u $ in $ W^{1,p} $.} \\
Fix any bounded domain $ E \Subset \Omega $. We want to prove that $ u_n \to u $ in $ W^{1,p}(E) $. Choose $ \psi \in C^\infty_c(\Omega) $ and a bounded domain $ F $ such that
\begin{linenomath}
\begin{equation*}
\psi \geq \chi_E \quad \mbox{and} \quad \supp \psi \Subset F \Subset \Omega.
\end{equation*}
\end{linenomath}
According to $ {\rm (a_2)} $, \eqref{pointconvgrad}, and \eqref{conv3}, by Fatou's lemma we deduce
\begin{linenomath}
\begin{equation*}
\begin{split}
\int_\Omega \psi a(\nabla u) \cdot \nabla u - m \int_E |\nabla u|^p &= \int_\Omega (\psi a(\nabla u) \cdot \nabla u - m|\nabla u|^p \chi_E) \\
&\leq \liminf_{n \to \infty} \int_\Omega (\psi a(\nabla u_n) \cdot \nabla u_n - m|\nabla u_n|^p \chi_E) \\
&= \int_\Omega \psi a(\nabla u) \cdot \nabla u - m \limsup_{n \to \infty} \int_E |\nabla u_n|^p,
\end{split}
\end{equation*}
\end{linenomath}
whence
\begin{equation}
\label{unifconv1}
\limsup_{n \to \infty} \int_E |\nabla u_n|^p \leq \int_E |\nabla u|^p.
\end{equation}
In addition, \eqref{boccardo2} implies
\begin{equation}
\label{unifconv2}
\lim_{n \to \infty} \int_E |u_n|^p = \int_E |u|^p.
\end{equation}
Hence \eqref{unifconv1}--\eqref{unifconv2} and the uniform convexity of $ W^{1,p}(E) $ lead to
\begin{linenomath}
\begin{equation*}
u_n \to u \quad \mbox{in} \;\; W^{1,p}(E). 
\end{equation*}
\end{linenomath}
\end{proof}
\begin{rmk}
For our purposes, in the sequel we will need only \eqref{pointconvgrad} and \eqref{limitprob}, but here we preferred to highlight the fact that convergence of gradients of solutions to \eqref{seqprob} is locally in $ L^p $, and not merely in $ L^q $ for any $ q < p $, as it occurs in the general framework of \cite{BM} (cf.\ (1.5) of \cite{BM}).
\end{rmk}
Strengthening the hypotheses of Lemma \ref{gradconvlocal} allows to prove the following global result.
\begin{lemma}
\label{gradconvglobal}
Let $ \Omega \subseteq \R^N $ be a bounded domain with smooth boundary and $ 1<s<p^* $. Let $ \{f_n\} \subseteq L^{s'}(\Omega) $ and $ \{u_n\} \subseteq W^{1,p}_0(\Omega) $ be such that
\begin{equation}
\label{seqprob2}
-\Div a(\nabla u_n) = f_n(x) \quad \mbox{in} \;\; W^{-1,p'}(\Omega),
\end{equation}
where $ a $ satisfies $ {\rm (a_1)} $--$ {\rm (a_3)} $. Suppose $ \{f_n\} $ and $ \{u_n\} $ to be bounded, respectively, in $ L^{s'}(\Omega) $ and $ W^{1,p}_0(\Omega) $. Then there exists $ f \in L^{s'}(\Omega) $ and $ u \in W^{1,p}_0(\Omega) $ such that $ f_n \rightharpoonup f $ in $ L^{s'}(\Omega) $ and $ u_n \to u $ in $ W^{1,p}_0(\Omega) $, up to subsequences. In particular, $ u $ solves
\begin{equation}
\label{limitprob2}
-\Div a(\nabla u) = f(x) \quad \mbox{in} \;\; W^{-1,p'}(\Omega).
\end{equation}
\end{lemma}
\begin{proof}
All the hypotheses of Lemma \ref{gradconvlocal} are satisfied, so \eqref{boccardo1}--\eqref{boccardo3} are guaranteed; due to the stronger hypotheses, they hold true with $ \Omega $ in place of $ E $. It is readily seen that also \eqref{test1}--\eqref{test2} holds for any $ \phi \in W^{1,p}_0(\Omega) $, without any restriction on $ \supp \phi $; in particular, this implies \eqref{limitprob2}, reasoning as in Lemma \ref{gradconvlocal}. Now we can test \eqref{seqprob2} with $ \phi = u_n $ and \eqref{limitprob2} with $ \phi = u $, obtaining
\begin{equation}
\label{adapt1}
\begin{split}
&\int_\Omega a(\nabla u_n) \cdot \nabla u_n = \int_\Omega f_n u_n, \\
&\int_\Omega a(\nabla u) \cdot \nabla u = \int_\Omega f u.
\end{split}
\end{equation}
Moreover, by the aforementioned considerations we deduce
\begin{equation}
\label{adapt2}
\lim_{n \to \infty} \int_\Omega f_n u_n = \int_\Omega f u.
\end{equation}
Arguing as in Lemma \ref{gradconvlocal}, \eqref{adapt1}--\eqref{adapt2} yield
\begin{linenomath}
\begin{equation*}
\lim_{n \to \infty} \int_\Omega a(\nabla u_n) \cdot \nabla u_n = \int_\Omega a(\nabla u) \cdot \nabla u.
\end{equation*}
\end{linenomath}
Thus Fatou's lemma, jointly with $ {\rm (a_2)} $ and \eqref{pointconvgrad}, gives
\begin{linenomath}
\begin{equation*}
\begin{split}
\int_\Omega a(\nabla u) \cdot \nabla u - m \int_\Omega |\nabla u|^p &= \int_\Omega (a(\nabla u) \cdot \nabla u - m |\nabla u|^p) \\
&\leq \liminf_{n \to \infty} \int_\Omega (a(\nabla u_n) \cdot \nabla u_n - m |\nabla u_n|^p) \\
&= \int_\Omega a(\nabla u) \cdot \nabla u - m \limsup_{n \to \infty} \int_\Omega |\nabla u_n|^p,
\end{split}
\end{equation*}
\end{linenomath}
ensuring
\begin{linenomath}
\begin{equation*}
\limsup_{n \to \infty} \int_\Omega |\nabla u_n|^p \leq \int_\Omega |\nabla u|^p.
\end{equation*}
\end{linenomath}
Uniform convexity of $ W^{1,p}_0(\Omega) $ then produces $ u_n \to u $ in $ W^{1,p}_0(\Omega) $.
\end{proof}

\section{Proof of the main result}

\subsection{Super-solution for regularized problems}

In this sub-section we suppose that $ a $ satisfies $ {\rm (a_1)} $--$ {\rm (a_3)} $, $ h,k $ are as in \eqref{hypf}--\eqref{hypg}, and \eqref{hyphkr} holds true. Consider the problem
\begin{equation}
\label{superprob}
\tag{$ {\rm \bar{P}} $}
\left\{
\begin{alignedat}{2}
-\Div a(\nabla u) &= h(x) + k(x)|\nabla u|^r \quad &&\mbox{in} \;\; \R^N, \\
u &\geq 0 \quad &&\mbox{in} \;\; \R^N.
\end{alignedat}
\right.
\end{equation}
This sub-section is devoted to prove existence of a distributional solution for \eqref{superprob}. We first study the following auxiliary problem, obtained by freezing the convection term in \eqref{superprob} and working on a ball $ B \subseteq \R^N $:
\begin{equation}
\label{frozenprob}
\tag{$ {\rm \bar{P}}_{v,B} $}
\left\{
\begin{alignedat}{2}
-\Div a(\nabla u) &= h(x) + k(x)|\nabla v(x)|^r \quad &&\mbox{in} \;\; B, \\
u &= 0 \quad &&\mbox{on} \;\; \partial B.
\end{alignedat}
\right.
\end{equation}

\begin{lemma}
\label{minty}
For any ball $ B \subseteq \R^N $ and any $ v \in W^{1,p}_0(B) $, problem \eqref{frozenprob} admits a unique weak solution $ u_v \in W^{1,p}_0(B) $. Moreover, $ u_v \geq 0 $ a.e.\ in $ B $.
\end{lemma}
\begin{proof}
Fix any $ B \subseteq \R^N $ and $ v \in W^{1,p}_0(B) $. According to the embedding inequality for $ L^{(p^*)'}(B) \hookrightarrow W^{-1,p'}(B) $, H\"older's inequality, and \eqref{hyphkr}, we have
\begin{linenomath}
\begin{equation*}
\begin{split}
\|h+k|\nabla v|^r\|_{W^{-1,p'}(B)} &\leq c \|h+k|\nabla v|^r\|_{L^{(p^*)'}(B)} \\
&\leq c (\|h\|_{L^{(p^*)'}(B)} + \|k|\nabla v|^r\|_{L^{(p^*)'}(B)}) \\
&\leq c(\|h\|_{L^\eta(\R^N)} + \|k\|_{L^\theta(\R^N)} \|\nabla v\|_{L^p(B)}^r),
\end{split}
\end{equation*}
\end{linenomath}
for a suitable $ c = c(|B|,N,p,r,\eta,\theta) > 0 $ changing its value from line to line. Hence, thanks to $ {\rm (a_1)} $--$ {\rm (a_3)} $ (see also Remark \ref{standardrmk}), Minty-Browder's theorem \cite[Theorem 5.16]{B} applies; so there exists a unique weak solution $ u_v \in W^{1,p}_0(B) $ to \eqref{frozenprob}. Since $ h,k \geq  0 $, the weak maximum principle guarantees $ u_v \geq 0 $ a.e.\ in $ B $.
\end{proof}
Keeping $ B $ fixed, we define the operator $ T:W^{1,p}_0(B) \to W^{1,p}_0(B) $ as
\begin{equation}
\label{Tdef}
T(v) = u_v,
\end{equation}
where $ u_v $ is given by Lemma \ref{minty}.
\begin{lemma}
\label{schauder}
For any ball $ B \subseteq \R^N $, the operator $ T $ defined in \eqref{Tdef} admits a fixed point $ u \in W^{1,p}_0(B) $. Moreover, $ u $ satisfies
\begin{equation}
\label{estimate}
\|\nabla u\|_{L^p(B)} \leq C,
\end{equation}
for a suitable constant $ C > 0 $ which is independent of $ B $.
\end{lemma}
\begin{proof}
According to \eqref{hyphkr}, by interpolation we have
\begin{linenomath}
\begin{equation*}
h \in L^{(p^*)'}(\R^N) \quad \mbox{and} \quad k \in L^\zeta(\R^N), \quad \mbox{with} \;\; \zeta = \left(\frac{1}{(p^*)'}-\frac{r}{p}\right)^{-1}.
\end{equation*}
\end{linenomath}
More precisely, the following estimates hold true (see \cite[Exercise 4.4]{B}):
\begin{equation}
\label{interpolation}
\begin{alignedat}{2}
&\|h\|_{L^{(p^*)'}(\R^N)} \leq \|h\|_{L^1(\R^N)}^{\lambda} \|h\|_{L^\eta(\R^N)}^{1-\lambda}, \quad &&\mbox{with} \;\; \lambda := \frac{\eta - (p^*)'}{\eta-1} \in (0,1), \\
&\|k\|_{L^{\zeta}(\R^N)} \leq \|k\|_{L^1(\R^N)}^{\mu} \|k\|_{L^\theta(\R^N)}^{1-\mu}, \quad &&\mbox{with} \;\; \mu := \frac{\theta - \zeta}{\theta-1} \in (0,1).
\end{alignedat}
\end{equation}
Testing \eqref{frozenprob} with $ u_v $, besides using $ {\rm (a_2)} $ as well as H\"older's and Sobolev's inequalities, gives
\begin{equation}
\label{energyest1}
\begin{split}
&m \|\nabla u_v\|_{L^p(B)}^p \leq \int_B a(\nabla u_v) \cdot \nabla u_v = \int_B (h + k|\nabla v|^r) u_v \\
&\leq (\|h\|_{L^{(p^*)'}(B)} + \|k\|_{L^\zeta(B)} \|\nabla v\|_{L^p(B)}^{r}) \|u_v\|_{L^{p^*}(B)} \\
&\leq c_S (\|h\|_{L^{(p^*)'}(\R^N)} + \|k\|_{L^\zeta(\R^N)} \|\nabla v\|_{L^p(B)}^{r}) \|\nabla u_v\|_{L^p(B)},
\end{split}
\end{equation}
where $ c_S = c_S(p,N) > 0 $ is the best constant of the inequality related to the Sobolev embedding $ W^{1,p}_0(B) \hookrightarrow L^{p^*}(B) $. We explicitly notice that $ c_S $ does not depend on $ B $; see \cite[p.290, Remark 20]{B}. Observe also that the estimates in \eqref{interpolation} are independent of $ B $. Dividing both sides of \eqref{energyest1} by $ m\|\nabla u_v\|_{L^p(B)} $ produces
\begin{equation}
\label{energyest2}
\begin{split}
&\|\nabla u_v\|_{L^p(B)}^{p-1} \leq c (1 + \|\nabla v\|_{L^p(B)}^r),
\end{split}
\end{equation}
being $ c = c(p,N,\|h\|_{(p^*)'},\|k\|_\zeta,m) > 0 $. Since $ r < p-1 $, we can choose $ C > 0 $ sufficiently large such that
\begin{equation}
\label{trapping1}
C^{p-1} \geq c(1+C^r),
\end{equation}
where $ c $ stems from \eqref{energyest2}. We define the closed, convex set
\begin{equation}
\label{trapping2}
K = \{u \in W^{1,p}_0(B): \, \|\nabla u\|_{L^p(B)} \leq C\},
\end{equation}
being $ C = C(p,N,r,\|h\|_{(p^*)'},\|k\|_\zeta,m) > 0 $ as in \eqref{trapping1}; actually, $ K $ is a ball in $ W^{1,p}_0(B) $. According to \eqref{energyest2}--\eqref{trapping2}, the set $ K $ is invariant under $ T $, that is,
\begin{equation}
\label{invariant}
T(K) \subseteq K.
\end{equation}

\noindent
\textit{\underline{Claim}: $ T $ is a completely continuous operator.} \\
First we prove compactness of $ T $. Take any bounded sequence $ \{v_n\} \subseteq W^{1,p}_0(B) $, set $ u_n := u_{v_n} $ for all $ n \in \N $, and reason up to subsequences. By reflexivity of $ W^{1,p}_0(B) $, there exists $ v \in W^{1,p}_0(B) $ such that
\begin{equation}
\label{vnweakconv}
v_n \rightharpoonup v \quad \mbox{in} \;\; W^{1,p}_0(B).
\end{equation}
Exploiting reflexivity again, \eqref{energyest2} and \eqref{vnweakconv} yield
\begin{equation}
\label{properties}
u_n \rightharpoonup u \quad \mbox{in} \;\; W^{1,p}_0(B), \quad u_n \to u \quad \mbox{in} \;\; L^p(B), \quad u_n \to u \quad \mbox{a.e.\ in} \;\; B,
\end{equation}
for some $ u \in W^{1,p}_0(B) $. According to \eqref{hyphkr}, we can choose $ s \in (1,p^*) $ satisfying \eqref{s}. Thus, by H\"older's inequality, we have
\begin{equation}
\label{unifbound}
\|h+k|\nabla v_n|^r\|_{L^{s'}(B)} \leq c(\|h\|_{L^\eta(\R^N)} + \|k\|_{L^\theta(\R^N)} \|\nabla v_n\|_{L^p(B)}^r)
\end{equation}
for a suitable $ c = c(|B|,p,r,s,\eta,\theta) > 0 $. The right-hand side of \eqref{unifbound} is uniformly bounded in $ n $, due to \eqref{vnweakconv}. Hence, we can apply Lemma \ref{gradconvglobal}, via \eqref{properties}--\eqref{unifbound}, to get $ u_n \to u $ in $ W^{1,p}_0(B) $. Since $ \{v_n\} $ was arbitrary, we deduce that $ T $ is a compact operator. \\
To prove continuity of $ T $, we suppose $ v_n \to v $ in $ W^{1,p}_0(B) $. Notice that \eqref{unifbound}, obtained under the weaker condition \eqref{vnweakconv}, holds true. Extracting a subsequence if necessary, we have $ \nabla v_n \to \nabla v $ a.e.\ in $ B $, so
\begin{equation}
\label{pointconv}
h+k|\nabla v_n|^r \to h+k|\nabla v|^r \quad \mbox{a.e.\ in} \;\; B.
\end{equation}
Using Proposition \ref{Brezis}, jointly with \eqref{unifbound}--\eqref{pointconv}, we get
\[
h+k|\nabla v_n|^r \rightharpoonup h+k|\nabla v|^r \quad \mbox{in} \;\; L^{s'}(B).
\]
Thus Lemma \ref{gradconvglobal} ensures that $ u $ is a weak solution to \eqref{frozenprob}, according to \eqref{limitprob2}. Uniqueness of solution to \eqref{frozenprob}, guaranteed by Lemma \ref{minty}, yields $ u = T(v) $. Arbitrariness of $ \{v_n\} $ provides the continuity of $ T $. Summarizing, $ T $ is a completely continuous operator, which proves the claim.

Recalling also \eqref{invariant}, Schauder's fixed point theorem \cite[Corollary 11.2]{GT} produces $ u \in K $ fixed point of $ T $; estimate \eqref{estimate} then follows by \eqref{trapping2}--\eqref{invariant}.
\end{proof}
Assuming, without loss of generality, $ x_0 = 0 $ in \eqref{hypf}, for all $ n \in \N $ we set
\begin{linenomath}
\begin{equation*}
B_n := B_n(x_0) = B_n(0).
\end{equation*}
\end{linenomath}
Let $ \tilde{u}_n \in W^{1,p}_0(B_n) $ be given by Lemma \ref{schauder} applied with $ B = B_n $. Each $ \tilde{u}_n $ can be extended to a (not relabeled) function in $ \mathcal{D}^{1,p}_0(\R^N) $ by setting $ \tilde{u}_n \equiv 0 $ outside $ B_n $; so \eqref{estimate} reads as
\begin{equation}
\label{estimate2}
\|\tilde{u}_n\|_{\mathcal{D}^{1,p}_0(\R^N)} = \|\nabla \tilde{u}_n\|_{L^p(\R^N)} \leq C \quad \forall n \in \N.
\end{equation}
Reflexivity of $ \mathcal{D}^{1,p}_0(\R^N) $ and \eqref{estimate2} provide $ \tilde{u} \in \mathcal{D}^{1,p}_0(\R^N) $ such that, up to subsequences,
\begin{equation}
\label{deftildeu}
\tilde{u}_n \rightharpoonup \tilde{u} \quad \mbox{in} \;\; \mathcal{D}^{1,p}_0(\R^N).
\end{equation}
Consider, for any $ j \in \N $, the restriction operator $ \Psi_j: \mathcal{D}^{1,p}_0(\R^N) \to L^p(B_j) $ defined as
\begin{linenomath}
\begin{equation*}
\Psi_j(u) = u_{\mid_{B_j}}.
\end{equation*}
\end{linenomath}
This operator is linear. Now we show that it is also continuous. Indeed, the H\"older and Sobolev inequalities imply
\begin{equation}
\label{estimate3}
\begin{split}
\|u\|_{W^{1,p}(B_j)} &= \|u\|_{L^{p}(B_j)} + \|\nabla u\|_{L^p(B_j)} \\
&\leq c \|u\|_{L^{p^*}(B_j)} + \|\nabla u\|_{L^p(B_j)} = c \|u\|_{\mathcal{D}^{1,p}_0(\R^N)}
\end{split}
\end{equation}
with $ c = c(j,p,N) > 0 $ varying at each passage; this proves that $ \Psi_j $ is continuous from $ \mathcal{D}^{1,p}_0(\R^N) $ to $ W^{1,p}(B_j) $. By linearity (see \cite[Theorem 3.10]{B}) $ \Psi_j $ is weakly continuous, so \eqref{deftildeu} entails $ \Psi_j(\tilde{u}_n) \rightharpoonup \Psi_j(\tilde{u}) $ in $ W^{1,p}(B_j) $. Then Rellich-Kondrachov's theorem guarantees $ \Psi_j(\tilde{u}_n) \to \Psi_j(\tilde{u}) $ in $ L^p(B_j) $. Hence \cite[Theorem 4.2]{B} and a diagonal argument produce
\begin{equation}
\label{supersol}
\tilde{u}_n \to \tilde{u} \quad \mbox{a.e.\ in} \;\; \R^N.
\end{equation}
\begin{lemma}
\label{supersollemma}
The function $ \tilde{u} \in \mathcal{D}^{1,p}_0(\R^N) $ defined in \eqref{deftildeu} is a distributional solution to \eqref{superprob}.
\end{lemma}
\begin{proof}
Take any $ \phi \in C^\infty_c(\R^N) $ and $ j \in \N $ such that $ \supp \phi \Subset B_j $. For each $ n > j $, Lemma \ref{schauder} and the definition of $ \tilde{u}_n $ ensure that
\begin{linenomath}
\begin{equation*}
-\Div a(\nabla \tilde{u}_n) = h+k|\nabla \tilde{u}_n|^r \quad \mbox{in} \;\; \mathscr{D}'(B_j).
\end{equation*}
\end{linenomath}
We take $ s \in (1,p^*) $ as in \eqref{s}, reason as for \eqref{unifbound}, and exploit \eqref{estimate2} to infer that $ \{h+k|\nabla \tilde{u}_n|^r\} $ is bounded in $L^{s'}(B_j) $. Hence, through \eqref{estimate2} and \eqref{estimate3} with $ u = \tilde{u}_n $, Lemma \ref{gradconvlocal} can be applied with $ \Omega = B_j $. In particular, for any domain $ E $ such that $ \supp \phi \Subset E \Subset B_j $ and for some $ \psi \in L^{s'}(E) $, we have
\begin{linenomath}
\begin{equation*}
\tilde{u}_n \to \tilde{u} \quad \mbox{in} \;\; W^{1,p}(E), \quad h+k|\nabla \tilde{u}_n|^r \rightharpoonup \psi \quad \mbox{in} \;\; L^{s'}(E).
\end{equation*}
\end{linenomath}
Another consequence of Lemma \ref{gradconvlocal} is the validity of \eqref{pointconvgrad}, that in our context reads as
\begin{linenomath}
\begin{equation*}
\nabla \tilde{u}_n \to \nabla \tilde{u} \quad \mbox{a.e.\ in} \;\; B_j.
\end{equation*}
\end{linenomath}
Thus Proposition \ref{Brezis} yields $ \psi = h+k|\nabla \tilde{u}|^r $. Equation \eqref{limitprob} holds with $ u = \tilde{u} $ and $ f = h+k|\nabla \tilde{u}|^r $; testing it with $ \phi $, besides recalling that $ \supp \phi \Subset B_j $, produces
\begin{linenomath}
\begin{equation*}
\int_{\R^N} a(\nabla \tilde{u}) \cdot \nabla \phi  = \int_{\R^N} (h+k|\nabla \tilde{u}|^r)\phi.
\end{equation*}
\end{linenomath}
Arbitrariness of $ \phi $ proves that $ \tilde{u} $ is a distributional solution to \eqref{superprob}. Since $ \tilde{u}_n \geq 0 $ and $ \tilde{u}_n \to \tilde{u} $ a.e.\ in $ \R^N $, according to Lemma \ref{minty} and \eqref{supersol} respectively, then $ \tilde{u} \geq 0 $ a.e.\ in $ \R^N $.
\end{proof}
Now suppose also \eqref{hypf}--\eqref{hypg} and fix any sequence $ \{\epsilon_n\} \subseteq (0,1) $ such that $ \epsilon_n \searrow 0 $. For all $ n \in \N $ consider the regularized problems
\begin{equation}
\label{ballprob}
\tag{$ {\rm P}_n $}
\left\{
\begin{alignedat}{2}
-\Div a(\nabla u) &= f(x,u+\epsilon_n) + g(x,\nabla u) \quad &&\mbox{in} \;\; B_n, \\
u &> 0 \quad &&\mbox{in} \;\; B_n, \\
u &= 0 \quad &&\mbox{on} \;\; \partial B_n.
\end{alignedat}
\right.
\end{equation}
If $ \tilde{u} \in W^{1,p}_{\rm loc}(\R^N) $, $ \tilde{u} \geq 0 $ a.e.\ in $ \R^N $, is a distributional solution to \eqref{superprob} (as, for instance, the one of Lemma \ref{supersollemma}), then the restriction of
\begin{equation}
\label{supersoldef}
\overline{u} := \tilde{u}+1 \in W^{1,p}_{\rm loc}(\R^N)
\end{equation}
to $ B_n $ is a super-solution to \eqref{ballprob}; indeed
\begin{equation}
\label{superboundbelow}
\overline{u} \geq 1 \quad \mbox{a.e.\ in} \;\; \R^N
\end{equation}
and, by \eqref{hypf}--\eqref{hypg} and a density argument,
\begin{linenomath}
\begin{equation*}
\begin{split}
&-\Div a(\nabla \overline{u}) = -\Div a(\nabla \tilde{u}) = h(x) + k(x)|\nabla \tilde{u}|^r \\
&\geq h(x)(\overline{u}+\epsilon_n)^{-\gamma} + k(x)|\nabla \overline{u}|^r \geq f(x,\overline{u}+\epsilon_n) + g(x,\nabla \overline{u})
\end{split}
\end{equation*}
\end{linenomath}
in weak sense (i.e., in $ W^{-1,p'}(B_n) $).

\subsection{Sub-solutions for regularized problems}

In this sub-section we construct, for any $ n \in \N $, a regular sub-solution $ \underline{u}_n $ to \eqref{ballprob}, and then solve \eqref{ballprob} via Theorem \ref{subsuperthm}.
\begin{lemma}
\label{subflemma}
Under \eqref{hypf}, there exists a continuous function $ \underline{f}: \R^N \times \R \to [0,+\infty) $ with the following properties:
\begin{itemize}
\item[$ {\rm (i)} $]\makebox[\textwidth][c]{$ \underline{f} \equiv \beta $ in $ B_\rho \times (-\epsilon,\epsilon) $ for some $ \beta, \rho, \epsilon > 0 $,}
\item[$ {\rm (ii)} $]\makebox[\textwidth][c]{$ \underline{f}(x,\cdot) $ is non-increasing in $ (0,+\infty) $ for all $ x \in \R^N $,}
\item[$ {\rm (iii)} $]\makebox[\textwidth][c]{$ \underline{f}(x,s) \leq f(x,s) $ for a.a.\ $ (x,s) \in \R^N \times (0,+\infty) $.}
\end{itemize}
\end{lemma}
\begin{proof}
From \eqref{hypf} we deduce that there exist $ \alpha,\beta > 0 $ such that
\begin{equation}
\label{liminf}
f(x,s) > \beta \quad \mbox{for a.a.} \;\; (x,s) \in B_\sigma \times (0,\alpha).
\end{equation}
Pick any $ \Theta \in C^\infty_c([0,+\infty)) $ such that
\begin{equation}
\label{cutoff}
\Theta(t) = \left\{
\begin{array}{ll}
1 \quad &\mbox{in} \;\; \left[0,\frac{1}{2}\right], \\
\mbox{decreasing} \quad &\mbox{in} \;\; \left(\frac{1}{2},1\right), \\
0 \quad &\mbox{in} \;\; [1,+\infty).
\end{array}
\right.
\end{equation}
According to \eqref{liminf}--\eqref{cutoff}, the function
\begin{linenomath}
\begin{equation*}
\underline{f}(x,s) = \beta \Theta \left( \frac{|x|}{\sigma} \right) \Theta \left( \frac{|s|}{\alpha} \right)
\end{equation*}
\end{linenomath}
satisfies the required conditions with $ \rho = \frac{\sigma}{2} $ and $ \epsilon = \frac{\alpha}{2} $: indeed, concerning (iii), notice that $ 0 \leq \underline{f} \leq \beta $ in $ \R^N \times \R $ and $ \underline{f}(x,s) = 0 $ if either $ x \notin B_\sigma $ or $ |s| > \alpha $.
\end{proof}
\begin{lemma}
\label{subsollemma}
Suppose $ {\rm (H_a)} $ and \eqref{hypf}. Then, for any $ n \in \N $, problem \eqref{ballprob} admits a sub-solution $ \underline{u}_n \in C^{1,\tau}(\overline{B}_n) $, for some $ \tau \in (0,1] $. Moreover, $ \underline{u}_n $ satisfies
\begin{equation}
\label{bound}
0 < \underline{u}_n \leq 1 \quad \mbox{in} \;\; B_n, \quad \underline{u}_n = 0 \quad \mbox{on} \;\; \partial B_n.
\end{equation}
\end{lemma}
\begin{proof}
Fix $ n \in \N $. For any $ \delta \in (0,1) $ consider the problem
\begin{equation}
\label{subprob}
\tag{$ {\rm \underline{P}}_{n,\delta} $}
\left\{
\begin{alignedat}{2}
-\Div a(\nabla u) &= f_\delta(x,u) \quad &&\mbox{in} \;\; B_n, \\
u &> 0 \quad &&\mbox{in} \;\; B_n, \\
u &= 0 \quad &&\mbox{on} \;\; \partial B_n,
\end{alignedat}
\right.
\end{equation}
where
\begin{equation}
\label{truncf}
f_\delta(x,s) := \min\{\underline{f}(x,s+\epsilon_1),\delta\},
\end{equation}
being $ \underline{f} $ as in Lemma \ref{subflemma}. Observe that
\begin{equation}
\label{apriori}
|f_\delta(x,s)| \leq \delta \leq 1 \quad \mbox{for a.a.} \;\; (x,s) \in \R^N \times \R, \quad \forall \delta \in (0,1).
\end{equation}
Since $ \epsilon_n \searrow 0 $, it is not restrictive to suppose $ \epsilon_1 < \epsilon $, with $ \epsilon $ stemming from Lemma \ref{subflemma}; hence Lemma \ref{subflemma} $ {\rm (i)} $ implies $ \underline{f}(x,\epsilon_1) = \beta > 0 $ for all $ x \in B_\rho $, whence
\begin{equation}
\label{nontrivial}
f_\delta(\cdot,0) \not\equiv 0 \quad \mbox{in} \;\; B_n.
\end{equation}
The energy functional $ J_{n,\delta}: W^{1,p}_0(B_n) \to \R $ associated to \eqref{subprob} is
\begin{linenomath}
\begin{equation*}
J_{n,\delta}(u) := \int_{B_n} A(\nabla u) \, {\rm d}x - \int_{B_n} F_\delta(x,u) \, {\rm d}x,
\end{equation*}
\end{linenomath}
being
\begin{linenomath}
\begin{equation*}
A(\xi) := \int_0^{|\xi|} ta_0(t) \, {\rm d}t, \quad F_\delta(x,s) := \int_0^{s} f_\delta(x,t) \, {\rm d}t.
\end{equation*}
\end{linenomath}
According to $ {\rm (a_1)} $--$ {\rm (a_2)} $ and \eqref{apriori}, standard arguments of Calculus of Variations guarantee that $ J_{n,\delta} $ is of class $ C^1 $, weakly sequentially lower semi-continuous, and coercive. Hence Weierstrass-Tonelli's theorem, together with \eqref{nontrivial}, furnishes $ \underline{u}_{n,\delta} \in W^{1,p}_0(B_n) $ non-trivial solution to \eqref{subprob}.

Lieberman's regularity (vide Remark \ref{standardrmk}), jointly with \eqref{apriori}, ensures that $ \{\underline{u}_{n,\delta}: \, \delta \in (0,1)\} $ is uniformly bounded in $ C^{1,\tau}(\overline{B}_n) $. Thus, Ascoli-Arzelà's theorem produces $ \underline{u}_{n,0} \in C^1(\overline{B}_n) $ such that $ \underline{u}_{n,\delta} \to \underline{u}_{n,0} $ in $ C^1(\overline{B}_n) $ as $ \delta \to 0^+ $. Passing to the limit in the weak formulation of \eqref{subprob} reveals that $ \underline{u}_{n,0} \equiv 0 $ in $ B_n $. Hence
\begin{equation}
\label{vanishing}
\lim_{\delta \to 0^+} \underline{u}_{n,\delta} = 0 \quad \mbox{in} \;\; C^1(\overline{B}_n).
\end{equation}
By means of \eqref{vanishing}, we can choose $ \delta_n > 0 $ such that
\begin{linenomath}
\begin{equation*}
|\underline{u}_{n,\delta_n}(x)| \leq 1 \quad \forall x \in B_n.
\end{equation*}
\end{linenomath}
Since $ \underline{u}_{n,\delta_n} $ is non-trivial, the strong maximum principle (see Remark \ref{standardrmk}) yields $ \underline{u}_{n,\delta_n} > 0 $ in $ B_n $. So \eqref{bound} holds true for
\begin{equation}
\label{subsoldef}
\underline{u}_n := \underline{u}_{n,\delta_n}.
\end{equation}
Using \eqref{truncf} and $ {\rm (ii)} $--$ {\rm (iii)} $ of Lemma \ref{subflemma}, besides $ \epsilon_n \searrow 0 $, leads to
\begin{linenomath}
\begin{equation*}
\begin{split}
-\Div a(\nabla \underline{u}_n) &= f_{\delta_n}(x,\underline{u}_n) \leq \underline{f}(x,\underline{u}_n+\epsilon_1) \leq \underline{f}(x,\underline{u}_n+\epsilon_n) \\
&\leq f(x,\underline{u}_n+\epsilon_n) \leq f(x,\underline{u}_n+\epsilon_n) + g(x,\nabla \underline{u}_n)
\end{split}
\end{equation*}
\end{linenomath}
in weak sense, proving that $ \underline{u}_n $ is a sub-solution to \eqref{ballprob}.
\end{proof}
\begin{lemma}
\label{subsuper}
Under $ {\rm (H_a)} $, \eqref{hypf}--\eqref{hypg}, and \eqref{hyphkr}, for any $ n \in \N $ problem \eqref{ballprob} admits a weak solution $ u_n \in W^{1,p}_0(B_n) $ satisfying $ \underline{u}_n \leq u_n \leq \overline{u} $, being $ \underline{u}_n $ and $ \overline{u} $ as in Lemma \ref{subsollemma} and \eqref{supersoldef} respectively.
\end{lemma}
\begin{proof}
Fix any $ n \in \N $. Observe that \eqref{bound} and \eqref{superboundbelow} yield $ \underline{u}_n \leq \overline{u} $ in $ B_n $, while \eqref{hypf}--\eqref{hypg} entail
\begin{linenomath}
\begin{equation*}
0 \leq f(x,s+\epsilon_n) + g(x,\xi) \leq h(x)\epsilon_n^{-\gamma} + k(x)|\xi|^r
\end{equation*}
\end{linenomath}
for a.a.\ $ x \in \Omega $ and all $ (s,\xi) \in [0,+\infty) \times \R^N $. The conclusion thus follows by applying Theorem \ref{subsuperthm}, after recalling \eqref{hyphkr} and \eqref{bound}.
\end{proof}

\subsection{Generalized solution to problem \eqref{prob}}

Hereafter we consider each $ u_n $ given by Lemma \ref{subsuper} as extended to the whole $ \R^N $ by setting $ u_n \equiv 0 $ outside $ B_n $; hence we can suppose $ u_n \in \mathcal{D}^{1,p}_0(\R^N) $. An analogous comment can be made for any $ \underline{u}_n $ arising from Lemma \ref{subsollemma}.

\begin{lemma}
\label{energyest}
Assume $ {\rm (H_a)} $, \eqref{hypf}--\eqref{hypg}, and \eqref{hyphkr}. Then, for any $ j \in \N $, there exists $ C_j > 0 $ such that
\begin{equation}
\label{mainenergyest}
\|u_n\|_{W^{1,p}(B_j)} \leq C_j \quad \forall n > j.
\end{equation}
\end{lemma}
\begin{proof}
Fix any $ j \in \N $. According to Lemma \ref{subsuper} we get
\begin{equation}
\label{potential}
\|u_n\|_{L^p(B_j)}^p \leq \|\overline{u}\|_{L^p(B_j)}^p.
\end{equation}

\noindent
\textit{\underline{Claim}: there exists $ \omega_j > 0 $ such that $ u_n \geq \omega_j $ a.e.\ in $ B_j $ for all $ n > j $.} \\
Pick any $ n > j $. Define $ \phi_n := (\underline{u}_{j+1}-u_n)_+ \in W^{1,p}_0(B_{j+1}) $ and extend it to a function of $ W^{1,p}_0(B_n) $ by setting $ \phi_n \equiv 0 $ in $ B_n \setminus B_{j+1} $. In this way $ \phi_n $ is an admissible test function for both $ ({\rm \underline{P}}_{j+1,\delta_{j+1}}) $ and \eqref{ballprob}; exploiting \eqref{subsoldef}, \eqref{truncf}, and $ {\rm (ii)} $--$ {\rm (iii)} $ of Lemma \ref{subflemma} we obtain
\begin{linenomath}
\begin{equation*}
\begin{split}
\int_{B_{j+1}} a(\nabla \underline{u}_{j+1}) \cdot \nabla \phi_n &\leq \int_{B_{j+1}} f_{\delta_{j+1}}(x,\underline{u}_{j+1}) \phi_n \leq \int_{B_{j+1}} \underline{f}(x,\underline{u}_{j+1}+\epsilon_1) \phi_n \\
&\leq \int_{B_{j+1}} \underline{f}(x,u_n+\epsilon_n) \phi_n \leq \int_{B_{j+1}} f(x,u_n+\epsilon_n) \phi_n \\
&\leq \int_{B_{j+1}} [f(x,u_n+\epsilon_n)+g(x,\nabla u_n)] \phi_n = \int_{B_{j+1}} a(\nabla u_n) \cdot \nabla \phi_n.
\end{split}
\end{equation*}
\end{linenomath}
Rearranging the terms gives
\begin{linenomath}
\begin{equation*}
\int_{B_{j+1} \cap \{\underline{u}_{j+1} > u_n\}} (a(\nabla \underline{u}_{j+1})-a(\nabla u_n)) \cdot (\nabla \underline{u}_{j+1} - \nabla u_n) \leq 0.
\end{equation*}
\end{linenomath}
Thus $ {\rm (a_3)} $ ensures $ \nabla \underline{u}_{j+1} = \nabla u_n $ a.e.\ in $ B_{j+1} \cap \{\underline{u}_{j+1} > u_n\} $, whence
\begin{equation}
\label{weakcomp}
\nabla \phi_n = 0 \quad \mbox{a.e.\ in} \;\; B_{j+1}.
\end{equation}
Recalling that $ \phi_n \in W^{1,p}_0(B_{j+1}) $, \eqref{weakcomp} yields $ \phi_n = 0 $ a.e.\ in $ B_{j+1} $, which means $ u_n \geq \underline{u}_{j+1} $ a.e.\ in $ B_{j+1} $. Since $ \underline{u}_{j+1} $ is continuous and strictly positive in $ B_j $ (cf.\ \eqref{bound}), we deduce that
\begin{equation}
\label{comparison}
u_n \geq \underline{u}_{j+1} \geq \min_{B_j} \underline{u}_{j+1} =: \omega_j > 0 \quad \mbox{a.e.\ in} \;\; B_j \quad \forall n > j,
\end{equation}
so the claim is proved. \\
Take any $ j \leq s < t \leq j+1 \leq n $, and then fix a cut-off function $ \psi \in C^\infty_c(\R^N) $ such that
\begin{equation}
\label{psiprops}
\begin{alignedat}{2}
&\psi \equiv 1 \quad \mbox{in} \;\; B_s, \quad &&\psi \equiv 0 \quad \mbox{in} \;\; \R^N \setminus B_t, \\
&\psi \geq 0 \quad \mbox{in} \;\; \R^N, \quad &&|\nabla \psi| \leq \frac{c_N}{t-s} \quad \mbox{in} \;\; \R^N.
\end{alignedat}
\end{equation}
Since $ (u_n-\omega_j)_+ \psi \in W^{1,p}_0(B_n) $, we can test \eqref{ballprob} with it; exploiting also $ {\rm (a_2)} $ we have
\begin{linenomath}
\begin{equation*}
\begin{split}
&m\int_{B_{j+1} \cap \{u_n>\omega_j\}} |\nabla u_n|^p \psi \\
&\leq \int_{B_{j+1} \cap \{u_n>\omega_j\}} \psi a(\nabla u_n) \cdot \nabla u_n = \int_{B_{j+1}} \psi a(\nabla u_n) \cdot \nabla(u_n-\omega_j)_+ \\
&= \int_{B_{j+1}} [f(x,u_n+\epsilon_n)+g(x,\nabla u_n)] (u_n-\omega_j)_+ \psi - \int_{B_{j+1}} (u_n-\omega_j)_+ a(\nabla u_n) \cdot \nabla \psi.
\end{split}
\end{equation*}
\end{linenomath}
Using \eqref{hypf}--\eqref{hypg}, $ {\rm (a_1)} $, and \eqref{psiprops} we arrive at
\begin{linenomath}
\begin{equation*}
\begin{split}
&m\int_{B_{j+1} \cap \{u_n>\omega_j\}} |\nabla u_n|^p \psi \\
&\leq \int_{B_{j+1} \cap \{u_n>\omega_j\}} h(u_n+\epsilon_n)^{-\gamma} u_n\psi + \int_{B_{j+1} \cap \{u_n>\omega_j\}} k|\nabla u_n|^r u_n \psi \\
&\quad + M \int_{B_t \cap \{u_n>\omega_j\}} (|\nabla u_n|^{p-1}+1) u_n |\nabla \psi| \\
&\leq \int_{B_{j+1} \cap \{u_n>\omega_j\}} hu_n^{1-\gamma} \psi + \int_{B_{j+1} \cap \{u_n>\omega_j\}} |\nabla u_n|^r\psi^{\frac{r}{p}} k u_n  \psi^{1-\frac{r}{p}} \\
&\quad + M \int_{B_t \cap \{u_n>\omega_j\}} (|\nabla u_n|^{p-1}+1) u_n |\nabla \psi|.
\end{split}
\end{equation*}
\end{linenomath}
Then, by \eqref{hyphkr} and \eqref{psiprops}, besides the Young, H\"older, and Sobolev inequalities, we get
\begin{linenomath}
\begin{equation*}
\begin{split}
&m\int_{B_{j+1} \cap \{u_n>\omega_j\}} |\nabla u_n|^p \psi \\
&\leq \omega_j^{1-\gamma} \|h\|_{L^1(\R^N)} + \frac{m}{4} \int_{B_{j+1} \cap \{u_n>\omega_j\}} |\nabla u_n|^p \psi + c \int_{B_{j+1} \cap \{u_n>\omega_j\}} k^{\frac{p}{p-r}} u_n^{\frac{p}{p-r}} \psi \\
&\quad + \frac{m}{4} \left( \int_{B_t \cap \{u_n>\omega_j\}} |\nabla u_n|^p + |B_{j+1}| \right) + c \int_{B_{j+1}} u_n^p |\nabla \psi|^p \\
&\leq \omega_j^{1-\gamma} \|h\|_{L^1(\R^N)} + \frac{m}{4} \int_{B_{j+1} \cap \{u_n>\omega_j\}} |\nabla u_n|^p \psi + c \|k\|_{L^\theta(\R^N)}\|\nabla \overline{u}\|_{L^p(B_{j+1})} \\
&\quad + \frac{m}{4} \int_{B_t \cap \{u_n>\omega_j\}} |\nabla u_n|^p + c \left( \frac{\|\overline{u}\|_{L^p(B_{j+1})}^p}{(t-s)^p} + 1 \right),
\end{split}
\end{equation*}
\end{linenomath}
for a suitable $ c = c(j,N,p,r,\theta,m) > 0 $ changing its value at each passage. Absorbing on the left-hand side the integral depending on $ \psi $, besides recalling \eqref{psiprops}, we get
\begin{linenomath}
\begin{equation*}
\begin{split}
&\int_{B_s \cap \{u_n>\omega_j\}} |\nabla u_n|^p \leq \int_{B_{j+1} \cap \{u_n>\omega_j\}} |\nabla u_n|^p \psi \\
&\leq \frac{1}{3} \int_{B_t \cap \{u_n>\omega_j\}} |\nabla u_n|^p + \frac{A}{(t-s)^p} + B,
\end{split}
\end{equation*}
\end{linenomath}
for suitable positive constants $ A $ (depending on $ j,N,p,r,\theta,m,\|\overline{u}\|_{L^p(B_{j+1})} $) and $ B $ (depending on $ j,N,p,r,\theta,m,\gamma,\|h\|_1, \|k\|_\theta,\|\nabla \overline{u}\|_{L^p(B_{j+1})} $). Hence \cite[Lemma 3.1, p.161]{G} can be applied to the function $ \Psi:[j,j+1] \to [0,+\infty) $ defined as
\begin{linenomath}
\begin{equation*}
\Psi(l) := \int_{B_l \cap \{u_n>\omega_j\}} |\nabla u_n|^p,
\end{equation*}
\end{linenomath}
producing
\begin{linenomath}
\begin{equation*}
\int_{B_j \cap \{u_n>\omega_j\}} |\nabla u_n|^p \leq c_p(A+B).
\end{equation*}
\end{linenomath}
Taking into account that $ B_j \subseteq \{u_n>\omega_j\} $, which is a consequence of \eqref{comparison}, we get
\begin{equation}
\label{kinetic}
\|\nabla u_n\|_{L^p(B_j)}^p \leq c_p(A+B).
\end{equation}
Adding \eqref{potential} and \eqref{kinetic} term by term produces \eqref{mainenergyest}.
\end{proof}

\begin{proof}[Proof of Theorem \ref{mainthm}]
Lemma \ref{energyest} ensures that, for any $ j \in \N $, the sequence $ \{u_n\}_{n>j} $ is bounded in $ W^{1,p}(B_j) $. Hence, a diagonal argument and the Rellich-Kondrachov theorem yield, up to subsequences,
\begin{equation}
\label{locconv}
u_n \rightharpoonup v^{(j)} \quad \mbox{in} \;\; W^{1,p}(B_j) \quad \mbox{and} \quad u_n \to v^{(j)} \quad \mbox{a.e.\ in} \;\; B_j
\end{equation}
for some $ v^{(j)} \in W^{1,p}(B_j) $, as $ n \to \infty $. We define
\begin{equation}
\label{solution}
u(x) = \left\{
\begin{array}{ll}
v^{(1)}(x) \quad &\mbox{in} \;\; B_1, \\
v^{(j)}(x) \quad &\mbox{in} \;\; B_j \setminus B_{j-1}, \;\; \forall j > 1.
\end{array}
\right.
\end{equation}
Uniqueness of limit and \eqref{locconv} yield $ v^{(i)} = v^{(j)} $ a.e.\ in $ B_i \cap B_j $ for all $ i,j \in \N $, so $ u = v^{(j)} $ a.e.\ in $ B_j $ for all $ j \in \N $. By construction we have $ u \in W^{1,p}_{\rm loc}(\R^N) $. It remains to prove that $ u $ is a generalized solution to \eqref{prob} (see Definition \ref{gensol}). \\
To this end, take any compact $ K \subseteq \R^N $ and fix $ j \in \N $ such that $ K \Subset B_j $. Passing to the limit in \eqref{comparison}, besides \eqref{locconv}--\eqref{solution}, gives $ u = v^{(j)} \geq \omega_j > 0 $ a.e.\ in $ B_j $, which ensures (i) of Definition \ref{gensol} by arbitrariness of $ K $. \\
Now take any $ \phi \in C^\infty_c(\R^N) $ and fix $ j \in \N $ such that $ \supp \phi \Subset B_j $. According to \eqref{ballprob}, each $ u_n $ with $ n > j $ satisfies
\begin{linenomath}
\begin{equation*}
-\Div a(\nabla u_n) = f(x,u_n+\epsilon_n) + g(x,\nabla u_n) \quad \mbox{in} \;\; \mathscr{D}'(B_j).
\end{equation*}
\end{linenomath}
Let $ s \in (1,p^*) $ as in \eqref{s}. By \eqref{hypf}--\eqref{hypg}, \eqref{comparison}, and H\"older's inequality we get
\begin{equation}
\label{brezisest}
\begin{split}
&\|f(x,u_n+\epsilon_n) + g(x,\nabla u_n)\|_{L^{s'}(B_j)}^{s'} \\
&\leq \int_{B_j} h^{s'}(u_n+\epsilon_n)^{-\gamma s'} + \int_{B_j} k^{s'}|\nabla u_n|^{rs'} \\
&\leq c(\omega_j^{-\gamma s'} \|h\|_{L^{\eta}(\R^N)}^{s'} + \|k\|_{L^\theta(\R^N)}^{s'} \|\nabla u_n\|_{L^p(B_j)}^{rs'})
\end{split}
\end{equation}
for some $ c = c(j,p,r,s,\eta,\theta) > 0 $. Since, due to \eqref{mainenergyest}, the right-hand side of \eqref{brezisest} is bounded, we infer that there exists $ \psi \in L^{s'}(B_j) $ such that
\begin{linenomath}
\begin{equation*}
f(x,u_n+\epsilon_n) + g(x,\nabla u_n) \rightharpoonup \psi \quad \mbox{in} \;\; L^{s'}(B_j).
\end{equation*}
\end{linenomath}
Lemma \ref{gradconvlocal} and \eqref{mainenergyest}, besides \eqref{locconv}--\eqref{solution}, ensure
\begin{equation}
\label{strongconv}
u_n \to u \quad \mbox{in} \;\; W^{1,p}(B_j)
\end{equation}
and
\begin{equation}
\label{finalprob}
-\Div a(\nabla u) = \psi \quad \mbox{in} \;\; \mathscr{D}'(B_j).
\end{equation}
An application of Proposition \ref{Brezis}, together with \eqref{brezisest}--\eqref{strongconv}, reveals that $ \psi = f(\cdot,u) + g(\cdot,\nabla u) $. Hence, testing \eqref{finalprob} with $ \phi $ yields
\begin{linenomath}
\begin{equation*}
\int_{\R^N} a(\nabla u) \cdot \nabla \phi = \int_{\R^N} [f(x,u) + g(x,\nabla u)] \phi.
\end{equation*}
\end{linenomath}
Arbitrariness of $ \phi $ proves (ii) of Definition \ref{gensol}, concluding the proof.
\end{proof}

\begin{rmk}
\label{variational}
If we consider $ \gamma \in (0,1) $ instead of $ \gamma \geq 1 $ in \eqref{hypf}, then we can do the same arguments: indeed, sub- and super-solutions do not depend on the explicit value of $ \gamma $, but their properties rely only on the fact that $ \gamma > 0 $, so that $ s \mapsto s^{-\gamma} $ is a decreasing function. We exploit $ \gamma \geq 1 $ only in the proof of Lemma \ref{energyest}, because we use
\begin{linenomath}
\begin{equation*}
\int_{B_{j+1} \cap \{u_n>\omega_j\}} h u_n^{1-\gamma} \psi \leq \omega_j^{1-\gamma} \|h\|_{L^1(\R^N)} \quad \mbox{since} \;\; \gamma \geq 1.
\end{equation*}
\end{linenomath}
But this estimate can be substituted, in the variational case, by
\begin{linenomath}
\begin{equation*}
\begin{split}
\int_{B_{j+1} \cap \{u_n>\omega_j\}} h u_n^{1-\gamma} \psi &\leq \int_{B_{j+1}} h\overline{u}^{1-\gamma} \leq \int_{B_{j+1}} h(\overline{u}+1) \\
&\leq \|h\|_{L^{(p^*)'}(B_{j+1})} \|\overline{u}\|_{L^{p^*}(B_{j+1})} + \|h\|_{L^1(\R^N)} \quad \quad \mbox{provided} \;\; \gamma \in (0,1), \\
&\leq c(\|h\|_{L^\eta(\R^N)} \|\nabla \overline{u}\|_{L^p(B_{j+1})} + \|h\|_{L^1(\R^N)})
\end{split}
\end{equation*}
\end{linenomath}
valid for a suitable $ c = c(j,N,p,\eta) > 0 $. Anyway, we deem that a simpler proof of Theorem \ref{mainthm} is available in the case $ \gamma \in (0,1) $, and it could be obtained by looking for solutions with finite energy, i.e., $ u \in \mathcal{D}^{1,p}_0(\R^N) $.
\end{rmk}

\appendix
\renewcommand{\thesection}{\Roman{section}}
\section{The differential operator}

Take any matrix $ a(\xi) = a_0(|\xi|) \xi $ (namely, having Uhlenbeck structure), where $ a_0: (0,+\infty) \to (0,+\infty) $ is a $ C^1 $ function. In this appendix we prove that hypotheses
\begin{itemize}
\item[$ {\rm (A_1)} $] \makebox[\textwidth][c]{$ \displaystyle{-1 < i_a := \inf_{t > 0} \frac{ta_0'(t)}{a_0(t)} \leq \sup_{t > 0} \frac{ta_0'(t)}{a_0(t)} =: s_a < +\infty,} $}
\item[$ {\rm (A_2)} $] \makebox[\textwidth][c]{$ mt^{p-1} \leq ta_0(t) \leq M(t^{p-1}+1) \quad \forall t \in (0,+\infty), $}
\end{itemize}
are equivalent to
\begin{itemize}
\item[$ {\rm (B_1)} $] \makebox[\textwidth][c]{$ (\nabla a(\xi) \mu) \cdot \mu \geq \frac{\omega(|\xi|)}{|\xi|} |\mu|^2 \quad \forall \xi,\mu \in \R^N, \, \xi \neq 0, $}
\item[$ {\rm (B_2)} $] \makebox[\textwidth][c]{$ |\nabla a(\xi)| \leq \Lambda \frac{\omega(|\xi|)}{|\xi|} \quad \forall \xi \in \R^N \setminus \{0\}, $}
\item[$ {\rm (B_3)} $] \makebox[\textwidth][c]{$ t \mapsto ta_0(t) $ is strictly increasing in $ (0,+\infty) $ and}
\begin{equation}
\label{redundant}
\qquad \qquad \qquad \lim_{s \to 0^+} ta_0(t) = 0, \quad \lim_{s \to 0^+} \frac{ta_0'(t)}{a_0(t)} > -1,
\end{equation}
\end{itemize}
where $ 1<p<+\infty $, $ m,M,\Lambda > 0 $, and $ \omega:(0,+\infty) \to (0,+\infty) $ is a $ C^1 $ function satisfying
\begin{equation}
\label{omega1}
C_1 \leq \frac{t\omega'(t)}{\omega(t)} \leq C_2 \quad \forall t \in (0,+\infty),
\end{equation}
and
\begin{equation}
\label{omega2}
C_3 t^{p-1} \leq \omega(t) \leq C_4(t^{p-1}+1) \quad \forall t \in (0,+\infty),
\end{equation}
for some $ C_1,C_2,C_3,C_4>0 $. In addition, we show that $ {\rm (B_3)} $ is redundant for the second set of hypotheses: more precisely, $ {\rm (B_1)} $--$ {\rm (B_2)} $ and \eqref{omega1} together imply $ {\rm (B_3)} $. We premit three considerations, concerning respectively the eigenvalues of $ \nabla a(\xi) $, an equivalent form of $ {\rm (B_1)} $--$ {\rm (B_2)} $, and the growth condition induced by \eqref{omega1}.

First of all, we compute
\begin{linenomath}
\begin{equation*}
\nabla a(\xi) = |\xi| a_0'(|\xi|) \frac{\xi}{|\xi|} \otimes \frac{\xi}{|\xi|} + a_0(|\xi|) I_N,
\end{equation*}
\end{linenomath}
where $ I_N $ stands for the $ N \times N $ identity matrix, while the symbol $ \otimes $ denotes the tensor product in $ \R^N $. In particular, $ \nabla a(\xi) $ is a symmetric matrix and possesses only the eigenvalues $ \lambda_1(\xi) = |\xi|a_0'(|\xi|) + a_0(|\xi|) $ (whose corresponding eigenspace is generated by $ \xi $) and $ \lambda_2(\xi) = a_0(|\xi|) $ (whose associated eigenspace is $ \xi^\perp $, that is, the orthogonal complement of $ \xi $). In particular, both eigenvalues are positive.

Secondly, we observe that $ {\rm (B_1)} $--$ {\rm (B_2)} $ are equivalent to
\begin{equation}
\label{equivalence}
\frac{\omega(|\xi|)}{|\xi|} \leq \lambda_{\rm min}(\xi) \leq \lambda_{\rm max}(\xi) \leq \Lambda \frac{\omega(|\xi|)}{|\xi|},
\end{equation}
where $ \lambda_{\rm min}(\xi) := \min\{\lambda_1(\xi),\lambda_2(\xi)\} $ and $ \lambda_{\rm max}(\xi) := \max\{\lambda_1(\xi),\lambda_2(\xi)\} $. Indeed, recalling that all norms on $ \R^N $ are equivalent, we can suppose that $ |\nabla a(\xi)| := |\nabla a(\xi)|_2 $; we also have $ |\nabla a(\xi)|_2 = \lambda_{\rm max} (\xi) $, since $ \nabla a(\xi) $ is symmetric. This fact shows that $ {\rm (B_1)} $ is equivalent to the last inequality of \eqref{equivalence}. On the other hand,
\begin{equation*}
\lambda_{\rm \min}(\xi)|\mu|^2 \leq (\nabla a(\xi) \mu) \cdot \mu \quad \mbox{for all} \;\; \mu \in \R^N
\end{equation*}
and
\begin{equation*}
\lambda_{\rm \min}(\xi)|\hat{\mu}|^2 = (\nabla a(\xi) \hat{\mu}) \cdot \hat{\mu} \quad \mbox{for some} \;\; \hat{\mu} \in \R^N
\end{equation*}
furnish the equivalence between $ {\rm (B_2)} $ and the first inequality of \eqref{equivalence}.

Finally, we notice that \eqref{omega1} forces a growth condition on $ \omega $ which differs from \eqref{omega2}: indeed, dividing \eqref{omega1} by $ t $ and integrating in $ [1,t] $ for any $ t \geq 1 $, one gets
\begin{linenomath}
\begin{equation*}
C_1 \log t \leq \log \omega(t) - \log \omega(1) \leq C_2 \log t,
\end{equation*}
\end{linenomath}
whence
\begin{equation}
\label{greaterthanone}
\omega(1)t^{C_1} \leq \omega(t) \leq \omega(1) t^{C_2} \quad \forall t \in [1,+\infty).
\end{equation}
Analogously, integrating in $ [t,1] $ for any $ t \in (0,1) $, one obtains
\begin{equation}
\label{lessthanone}
\omega(1)t^{C_2} \leq \omega(t) \leq \omega(1) t^{C_1} \quad \forall t \in (0,1).
\end{equation}
By \eqref{greaterthanone}--\eqref{lessthanone} we deduce
\begin{equation}
\label{nonstandard}
\omega(1)\min\{t^{C_1},t^{C_2}\} \leq \omega(t) \leq \omega(1)\max\{t^{C_1},t^{C_2}\} \quad \forall t \in (0,+\infty).
\end{equation}
It is worth noticing that, due to \eqref{nonstandard}, hypothesis \eqref{omega1} forces $ p \in [C_1+1,C_2+1] $, where $ p $ stems from \eqref{omega2}. Now we are ready to prove the equivalence between $ {\rm (A_1)} $--$ {\rm (A_2)} $ and $ {\rm (B_1)} $--$ {\rm (B_3)} $.

\noindent
\textit{\underline{Step 1}: $ {\rm (A_1)} $ implies $ {\rm (B_1)} $--$ {\rm (B_2)} $ for a suitable $ \omega $ satisfying \eqref{omega1}.} \\
Choose
\begin{equation}
\label{choice}
\omega(t) := kta_0(t) \quad \forall t \in (0,+\infty),
\end{equation}
with $ k := \min\{1,i_a+1\} $. Hypothesis $ {\rm (A_1)} $ yields, for all $ t > 0 $,
\begin{equation}
\label{elliptic}
\frac{t\omega'(t)}{\omega(t)} = \frac{ta_0'(t)+a_0(t)}{a_0(t)} = \frac{ta_0'(t)}{a_0(t)} + 1 \in [i_a+1,s_a+1],
\end{equation}
so \eqref{omega1} is proved with $ C_1 := i_a+1 $ and $ C_2 := s_a+1 $. Observe that
\begin{linenomath}
\begin{equation*}
\lambda_1(\xi) = a_0(|\xi|) \left( \frac{|\xi|a_0'(|\xi|)}{a_0(|\xi|)}+1 \right) = \frac{\omega(|\xi|)}{k|\xi|} \left( \frac{|\xi|a_0'(|\xi|)}{a_0(|\xi|)}+1 \right)
\end{equation*}
\end{linenomath}
and
\begin{linenomath}
\begin{equation*}
\lambda_2(\xi) = a_0(|\xi|) = \frac{\omega(|\xi|)}{k|\xi|}.
\end{equation*}
\end{linenomath}
Thus, \eqref{elliptic} implies \eqref{equivalence} with $ \Lambda := k^{-1}\max\{1,s_a+1\} $.

\noindent
\textit{\underline{Step 2}: $ {\rm (A_2)} $ implies \eqref{omega2} for the $ \omega $ defined in \eqref{choice}.} \\
It suffices to choose $ C_3 := km $ and $ C_4 := kM $, since we have
\begin{linenomath}
\begin{equation*}
kmt^{p-1} \leq \omega(t) \leq kM(t^{p-1}+1) \quad \forall t \in (0,+\infty).
\end{equation*}
\end{linenomath}

\noindent
\textit{\underline{Step 3}: $ {\rm (B_1)} $--$ {\rm (B_2)} $ imply $ {\rm (A_1)} $.} \\
Fix $ t > 0 $ and pick any $ \xi \in \R^N $ such that $ |\xi| = t $. Notice that
\begin{equation}
\label{eigenratio}
\frac{ta_0'(t)}{a_0(t)} = \frac{\lambda_1(\xi)}{\lambda_2(\xi)} - 1.
\end{equation}
By \eqref{equivalence} we get
\begin{equation}
\label{elliptbound}
\frac{1}{\Lambda} \leq \frac{\lambda_{\rm min}(\xi)}{\lambda_{\rm \max}(\xi)} \leq \frac{\lambda_1(\xi)}{\lambda_2(\xi)} \leq \frac{\lambda_{\rm max}(\xi)}{\lambda_{\rm \min}(\xi)} \leq \Lambda,
\end{equation}
whence
\begin{linenomath}
\begin{equation*}
-1 < \frac{1}{\Lambda}-1 \leq \frac{ta_0'(t)}{a_0(t)} \leq \Lambda-1 < +\infty.
\end{equation*}
\end{linenomath}
Arbitrariness of $ t $ ensures $ {\rm (A_1)} $.

\noindent
\textit{\underline{Step 4}: $ {\rm (B_1)} $--$ {\rm (B_2)} $ and \eqref{omega2} jointly imply $ {\rm (A_2)} $.} \\
Consider $ t,\xi $ as in Step 3. According to \eqref{equivalence} and \eqref{omega2}, besides recalling that $ \lambda_2(\xi) = a_0(t) $, we have
\begin{linenomath}
\begin{equation*}
C_3 t^{p-1} \leq \omega(t) \leq ta_0(t) \leq \Lambda \omega(t) \leq \Lambda C_4(t^{p-1}+1),
\end{equation*}
\end{linenomath}
so $ {\rm (A_2)} $ is satisfied with $ m := C_3 $ and $ M =: \Lambda C_4 $. 

\noindent
\textit{\underline{Step 5}: $ {\rm (B_1)} $--$ {\rm (B_2)} $ and \eqref{omega1} jointly imply $ {\rm (B_3)} $.} \\
Strict monotonicity of $ t \mapsto ta_0(t) $ comes from \eqref{eigenratio}--\eqref{elliptbound}: indeed,
\begin{linenomath}
\begin{equation*}
(ta_0(t))' = ta_0'(t)+a_0(t) = a_0(t) \left( \frac{ta_0'(t)}{a_0(t)}+1 \right) \geq \frac{1}{\Lambda} a_0(t) > 0 \quad \forall t \in (0,+\infty).
\end{equation*}
\end{linenomath}
The same argument reveals that
\begin{linenomath}
\begin{equation*}
\lim_{t \to 0^+} \frac{ta_0'(t)}{a_0(t)} \geq \frac{1}{\Lambda}-1 > -1,
\end{equation*}
\end{linenomath}
which implies the second part of \eqref{redundant}. By \eqref{equivalence} and \eqref{lessthanone} one has
\begin{linenomath}
\begin{equation*}
0 \leq ta_0(t) \leq \Lambda \omega(t) \leq \Lambda \omega(1)t^{C_1} \quad \forall t \in (0,1),
\end{equation*}
\end{linenomath}
so letting $ t \to 0^+ $ entails $ ta_0(t) \to 0 $, concluding the proof of \eqref{redundant}.

\begin{rmk}
\label{finalrmk}
Hypothesis \eqref{omega1} is redundant in this sense: $ {\rm (B_1)} $--$ {\rm (B_2)} $, written for $ \omega $ not necessarily satisfying \eqref{omega1}, imply $ {\rm (A_1)} $ (see Step 3 above), and $ {\rm (A_1)} $ in turn implies $ {\rm (B_1)} $--$ {\rm (B_2)} $ for a suitable $ \tilde{\omega} $ instead of $ \omega $ (cf.\ \eqref{choice}); moreover, $ \tilde{\omega} $ fulfills \eqref{omega1}, from Step 1. In other words, under conditions $ {\rm (B_1)} $--$ {\rm (B_2)} $, it is possible to replace $ \omega $ with another $ \tilde{\omega} $ such that $ \tilde{\omega} $ obeys \eqref{omega1}. If, in addition, \eqref{omega2} is satisfied by $ \omega $, then also $ \tilde{\omega} $ satisfies \eqref{omega2} (vide Steps 4 and 2).
\end{rmk}

\section*{Acknowledgments}

\noindent
The authors thank Professor S.A. Marano for suggesting this research topic. \\
The second author is supported by the following research projects: 1) PRIN 2017 `Nonlinear Differential Problems via Variational, Topological and Set-valued Methods' (Grant No. 2017AYM8XW) of MIUR; 2) PRA 2020-2022 Linea 3 `MO.S.A.I.C.' of the University of Catania.

\begin{small}

\end{small}

\end{document}